\documentclass[12pt,letterpaper,reqno]{amsart}
\usepackage{epsfig}
\usepackage{amsmath}
\usepackage{amssymb}
\usepackage{amsthm}
\usepackage{indentfirst}
\usepackage{xspace}
\usepackage[dvipsnames]{xcolor}
\usepackage{setspace}
\usepackage{verbatim}
\usepackage[letterpaper,margin=1in,headheight=15pt]{geometry}
\usepackage{mathpazo}
\usepackage{booktabs}
\usepackage{float} 
\definecolor{shadecolor}{rgb}{0.85,0.85,0.85}
\usepackage{bibentry}
\usepackage{bm}
\usepackage{hyperref}
\definecolor{darkred}{rgb}{0.5,0.15,0.15}
\hypersetup{colorlinks=true,urlcolor=darkred,linkcolor=darkred,citecolor=darkred}
\allowdisplaybreaks

\allowdisplaybreaks

\newtheorem{thm}{Theorem}
\newtheorem{cor}[thm]{Corollary}

\newtheorem{lem}[thm]{Lemma}
\newtheorem{prop}[thm]{Proposition}

\newtheorem{ex}[thm]{Example}

\newtheorem{construction}[thm]{Construction}

 \newtheorem{mainthm}[thm]{Main Theorem}
 \newtheorem{defprop}[thm]{Definition/Proposition}  
\theoremstyle{remark}
\newtheorem{rem}[thm]{Remark}
\theoremstyle{definition}
\newtheorem{defn}[thm]{Definition}

\numberwithin{thm}{section}
\numberwithin{equation}{section}
\numberwithin{figure}{section}

\newcommand{\cB}{\ensuremath{\mathcal B}}
\newcommand{\cL}{\ensuremath{\mathcal L}}

\newcommand{\cM}{\ensuremath{\mathcal M}}

\newcommand{\R}{\ensuremath{\mathbb R}}
\newcommand{\C}{\ensuremath{\mathbb C}}
\newcommand{\CP}{\ensuremath{\mathbb {CP}}}

\newcommand{\Z}{\ensuremath{\mathbb Z}}

\newcommand{\bD}{{\mathbb{D}}}

\newcommand{\cE}{{\mathcal E}}

\newcommand{\delbar}{\ensuremath{\overline{\partial}}}
\newcommand{\zbar}{\ensuremath{\overline{z}}}

\newcommand{\Id}{\ensuremath{\mathrm{Id}}}

\newcommand{\app}{\ensuremath{\mathrm{app}}}
\newcommand{\model}{\ensuremath{\mathrm{mod}}}

\newcommand{\I}{{\mathrm i}}
\newcommand{\e}{{\mathrm e}}
\newcommand{\de}{\mathrm{d}}

\newcommand{\norm}[1]{\lVert#1\rVert}
\newcommand{\IP}[1]{\langle#1\rangle}

\newcommand{\eps}{\epsilon}

\newcommand{\del}{{\partial}}

\DeclareMathOperator{\Det}{Det}
\DeclareMathOperator{\End}{End}

\DeclareMathOperator{\Aut}{Aut}

\newcommand{\fixme}[1]{{\color{blue}{[#1]}}}

 \def\deg{\mathrm{deg}\,}
\def\fid{{\mathrm{mod}}}
\def\app{{\mathrm{app}}}
 \def\2{{(2)}}

 \def\rk{\mathrm{rank}}
 
 \def\<{\left\langle}                
 \def\>{\right\rangle}                 
  
 \def\rext{\mathrm{ext}}

\title{Generic Ends of the Moduli Space of $\mathrm{SL(n,\C)}$-Higgs Bundles}
\date{}

\author{Laura Fredrickson}

\begin{document}
\onehalfspacing
\maketitle

\begin{abstract}
Given a generic ray of Higgs bundles $(\delbar_E, t\varphi)$,
we describe the corresponding family of hermitian metrics $h_t$ solving Hitchin's equations via gluing methods.  
In the process, we construct a family of approximate solutions $h_t^\app$ which differ from  the actual harmonic metrics $h_t$ by error terms of size $\e^{-\delta t}$. Such families of explicit approximate solutions have already proved useful for answering finer questions about the asymptotic geometry of the Hitchin moduli space.
\end{abstract}

\section{Introduction}

In this paper, we describe solutions of Hitchin's equations near the generic ends of the $SU(n)$-Hitchin moduli space by constructing good approximate solutions and perturbing them to actual solutions.  Our paper generalizes Mazzeo-Swoboda-Weiss-Witt's results for the $SU(2)$-Hitchin moduli space \cite{MSWW14, MSWW15}, which has already been useful in their more recent work of the asymptotic geometry of the Hitchin moduli space \cite{MSWW17}. Gaiotto-Moore-Neitzke give a conjectural description of the hyperk\"ahler metric on the $SU(n)$-Hitchin moduli space \cite{GMNhitchin, GMNwallcrossing}, and our finer description
of solutions of $SU(n)$-Hitchin's equations near the ends is a first step towards proving their conjecture. In fact, a number of conjectures from mathematicians and physics about $\cM$ remain open because they require a finer knowledge of the ends of the moduli space
than provided by traditional algebro-geometric techniques alone. As demonstrated in \cite{MSWW17}, constructive analytic techniques complement these well, so we take this approach.

\subsection{Fixed data} \label{sec:fixeddata}

Fix $C=C(I, g_C, \omega)$ a compact K\"ahler curve of $\mathrm{genus} \geq 2$
with metric $g_C$, complex structure $I$, and symplectic form $\omega_C$.
Let $K_C$ be the canonical line bundle. 
Fix $E \rightarrow C$ a complex vector bundle of rank $n$ and degree $d$. 
Let  $\Det \,E$ be the determinant line bundle.
The groups $\Aut(E)$ and $\End(E)$ respectively denote the automorphisms and endomorphisms
of the complex vector bundle $E$ which induce the identity map on $\Det E$.

Additionally fix a 
holomorphic structure, $\delbar_{\Det E}$, and a hermitian structure, $h_{\Det E}$, on the complex line bundle $\Det E$
such that
$h_{\Det E}$ is Hermitian-Einstein for the holomorphic line bundle $(\Det E, \delbar_{\Det E})$.
We do not normalize the Riemannian volume $\mathrm{vol}_{g_C}(C)$ of the curve $C$.
Consequently,
the Hermitian-Einstein condition states that the curvature of the associated Chern connection $D=D(\delbar_{\Det E}, h_{\Det E})$ 
satisfies 
 \begin{equation}
 F_{D} = -\sqrt{-1} \mathrm{deg} E 
  \frac{2\pi \omega_C}{\mathrm{vol}_{g_C}(C)} \Id_{\Det E}.
\end{equation}

\medskip

Given this fixed data, let $\cM$ be the associated Hitchin moduli space.
The Hitchin moduli space consists of triples $(\delbar_E, \varphi, h)$ solving Hitchin's equations up to complex gauge equivalence, defined in \eqref{eq:complexgaugeequiv}. Here,
\begin{itemize}
 \item $\delbar_E$ is a holomorphic structure on $E$,
 \item   $\varphi \in \Omega^{1,0}(C, \End E)$ is the Higgs field, and  
 \item  $h$ is a hermitian metric on $E$. 
\end{itemize}
Additionally, the induced holomorphic and hermitian structures on $\Det E$ must agree with the fixed structures $\delbar_{\Det E}$ and $h_{\Det E}$.  
We say that such a triple $(\delbar_E, \varphi, h)$ is a \emph{solution of $SU(n)$-Hitchin's equations} if 
 \begin{equation} \label{eq:Hitchin}
 \delbar_E \varphi =0, \qquad F^\perp_{D(\delbar_E,h)} + [\varphi, \varphi^{\dagger_h}]=0,
\end{equation}
 where $D(\delbar_E, h)$ is the Chern connection,  $\varphi^{\dagger_h} \in \Omega^{0,1}(C, \End E)$ is the $h$-hermitian adjoint, and $F_D^\perp$ denotes the trace-free part of the the curvature of $D$, i.e. 
 \begin{equation}
  F_{D(\delbar, h)}^\perp= F_{D(\delbar_E,h)} +\sqrt{-1} \frac{\deg(E)}{\rk(E)}  \frac{2\pi \omega_C}{\mathrm{vol}_{g_C}(C)} \Id_E.
 \end{equation}
We call such $h$ the \emph{harmonic metric}
 for the Higgs bundle $(\delbar_E, \varphi)$. A Higgs bundle $(\delbar_E, \varphi)$ admits a harmonic metric if, and only if, $(\delbar_E,\varphi)$ is polystable\footnote{
A Higgs bundle $(\delbar_E, \varphi)$ is \emph{stable} if
for all $\varphi$-invariant subbundles $F$,
$\mu(F) < \mu(E)$; here, $\mu(F):=\frac{\mathrm{deg}(F)}{\mathrm{rank}(F)}$
is the slope of the bundle.
A Higgs bundle is \emph{polystable} if it is the direct sum
of stable Higgs bundles of the same slope.
}.

The action of the complex gauge group $\Aut(E)$ is as follows: given $g\in \Aut(E)$, 
\begin{equation} \label{eq:complexgaugeequiv}
 g \cdot (\delbar_E, \varphi, h) = ( g^{-1} \circ \delbar_E \circ g, g^{-1} \varphi g, g \cdot h), \qquad \mbox{where $(g \cdot h)(v,w)=h(g v,g w)$}.
\end{equation}

\bigskip

\noindent\emph{Unitary formulation of Hitchin's equations.} There's an equivalent unitary formation of Hitchin's equations. In this formulation, we additionally fix a hermitian metric $h_0$ on the complex vector bundle $E \to C$. Now, the Hitchin moduli space consists of pairs $(\de_A, \Phi)$, where
\begin{itemize}
 \item $\de_A$ is a $h_0$-unitary connection, and
 \item $\Phi \in \Omega^{1,0}(\C, \End E)$, 
\end{itemize}
solving $\delbar_A \Phi=0$ and $F^\perp_A + [\Phi, \Phi^{\dagger_{h_0}}]=0$---up to $h_0$-unitary gauge equivalence.

We can pass back and between these two formulations.
Given the pair $(\de_A, \Phi)$, we get the associated triple $(\delbar_A, \Phi, h_0)$. Conversely, given a triple $(\delbar_E, \varphi, h)$,
there is an $\End E$-valued $h_0$-hermitian section $H$ such that $h(v,w)=h_0(Hv, w)$.
 Take the complex gauge transformation $g=H^{-1/2}$.
 Observe that in general, $(g \cdot h)(v,w)=h_0((g^{\dagger_{h_0}} H g) v, w)$;
 consequently, for our choice of gauge transformation $g=H^{-1/2}$, indeed $(g \cdot h)=h_0$.
 Then, for the complex gauge action in \eqref{eq:complexgaugeequiv},
 $g \cdot (\delbar_E, \varphi, h)=(H^{1/2} \circ \delbar_E \circ H^{-1/2}, H^{1/2} \varphi H^{-1/2}, h_0)$. Consequently, 
the associated pair $(\de_A, \Phi)$ is defined by $\delbar_A = H^{1/2} \circ \delbar_E \circ H^{-1/2}$ and $\Phi = H^{1/2} \varphi H^{-1/2}$.

\begin{rem}
Locally, it will be convenient to work in both holomorphic and unitary gauges. 
A local basis $\{s_1, \cdots, s_n\}$ of sections of $E$ is holomorphic if $s_i$
are holomorphic sections of $(E, \delbar_E)$.  
A local basis $\{s_1, \cdots, s_n\}$ of sections of $E$ is unitary if $h_0(s_i, s_j) = \delta_{ij}$.
\end{rem}

\subsection{Summary of results}\label{sec:results}

Fix a polystable Higgs bundle $(\delbar_E, \varphi)$ in a non-degenerate fiber of the Hitchin fibration $\cM \rightarrow \cB$.
Consider the $\R^{+}_t$-family of Higgs bundles 
$(\delbar_E, t \varphi)$. 
We seek to describe the corresponding family of harmonic
metric $h_t$ for $t\gg0$, as shown in Figure \ref{fig:path}. 
To describe  $h_t$,
\begin{itemize}
\item first, we construct a singular hermitian metric $h_\ell$ (\emph{A posteriori}, we find in Corollary \ref{cor:main} that $h_\ell = \lim_{t \to \infty} h_t$.) (\S\ref{sec:limitingconstruction});
\item then, we construct a family of approximate solutions, $h_t^\app$, built from the singular hermitian metric $h_\ell$ and a family of local model solutions (\S\ref{sec:modelsolutions});
\item we prove that this family of approximate hermitian metrics $h_t^\app$
solves Hitchin's equations up to an exponentially-decaying error
(Proposition \ref{defpropreg}); and
\item finally, we perturb from the approximate solutions $h_t^\app$ to the actual solutions $h_t$
using a contraction mapping argument.
This last point is the content of the main theorem, Theorem \ref{thm:main}.
\end{itemize}

\begin{figure}[ht]
 \begin{centering}
  \includegraphics[height=1in]{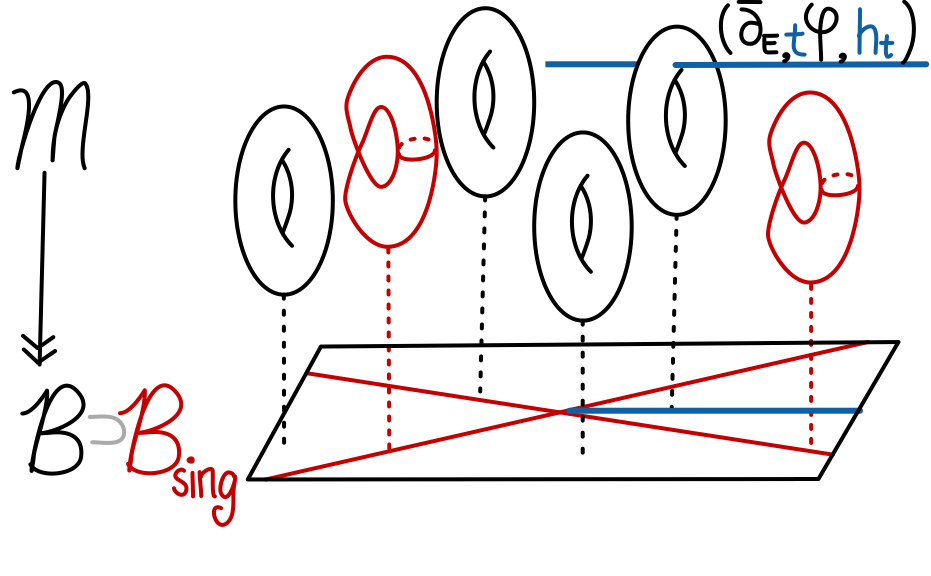}\\
  \caption{\label{fig:path} \emph{
A $\R^{+}$-family of Higgs bundles approaching
the ($t\!=\!\infty$)-ends of the Hitchin moduli space.}}
   \end{centering}
\end{figure}

The strategy of the proof outlined above is the same as the strategy in the $n=2$ case appearing in \cite{MSWW14, MSWW15}.
We highlight some  notable differences.
First, in the $n=2$ case, the singular hermitian metric $h_\ell$ could be desingularized
using a single $2 \times 2$ model solution. In the rank $n$ case,
we need $K \times K$ model solutions for $K=2, \cdots, n$ to desingularize $h_\ell$.
We discuss these model solutions in \S\ref{sec:MK1}.
Secondly, the proof that the inverse of the linearized operator is bounded (Proposition \ref{prop:boundoninverse})
requires substantial modification from \cite{MSWW14}. (In Remark \ref{rem:domaindecomp}, we make a lengthy remark about why Mazzeo-Swoboda-Weiss-Witt's method does not work if $n > 2$.)

\begin{rem} \label{rem:metric}
Note that in both Mazzeo-Swoboda-Weiss-Witt's proof in \cite{MSWW14, MSWW15} and the one here,
we use the fact that Hitchin's equations are conformal.  After fixing a polystable Higgs bundle $(\delbar_E, \varphi)$ in a non-degenerate fiber of $\cM \to \cB$, we take a conformal metric $g_C'$
on $C$ which is flat on disks around the zeros of the discriminant section $\Delta_\varphi$ defined in \eqref{eq:discriminant}. The convenience of the metric $g_C'$ will be discussed further in \S\ref{sec:approximatesolutions}.
\end{rem}
\subsection{Acknowledgements}

Many thanks to Rafe Mazzeo and Andy Neitzke for helpful discussions. 

\section{Higgs bundles in \texorpdfstring{$\cM'$}{M'}} \label{sec:simple}

\subsection{Review of Hitchin fibration}

The $SU(n)$-Hitchin moduli space $\cM$ is a complex integrable system with half-dimensional base $\mathcal{B}$. The Hitchin fibration is
\begin{eqnarray}
 \mathrm{Hit}:  \qquad   \cM  \quad & \rightarrow & \cB\\ \nonumber
 (\delbar_E, \varphi, h) &\mapsto& \mathrm{char}_\varphi(\lambda),
\end{eqnarray}
where $\mathrm{char}_\varphi(\lambda)$ is the characteristic polynomial of $\varphi \in \Omega^{1,0}(C, \End E)$. The Hitchin base $\mathcal{B}$ can be identified with the complex vector space $\oplus_{i=2}^n H^0(C, K_C^i) \ni \mathbf{b}=(q_2, \cdots, q_n)$, under the map from $\mathrm{char}_\varphi (\lambda)$ to its coefficients
\begin{equation}
 \mathrm{char}_\varphi (\lambda) = \lambda^n + q_2 \lambda^{n-2} + \cdots + q_{n-1} \lambda + q_n.
\end{equation}
A point $\mathbf{b} \in \cB$ encodes the eigenvalues of $\varphi$. We can geometrically
package the eigenvalues as a ramified $n\!:\!1$-cover
cut out of the total space of holomorphic cotangent bundle $K_C \rightarrow C$ by the equation
\begin{equation}\label{eq:spectralcover}
 \Sigma 
=\{ \lambda \in K_C: \mathrm{char}_\varphi(\lambda) = 0 \}.
\end{equation}
Call $\Sigma \overset{\pi}{\rightarrow} C$ the \emph{spectral cover}.

\medskip

The fiber $\mathrm{Hit}^{-1}(\mathbf{b})$ is a compact abelian variety
if, and only if, the spectral curve $\Sigma_{\mathbf{b}}$ is smooth.
Let $\cB'$ be this locus where the spectral cover $\Sigma_{\mathbf{b}}$ is smooth.
We
restrict our attention to Higgs bundles in the \emph{regular locus} $\cM'=\mathrm{Hit}^{-1}(\cB')$, and call such Higgs bundles \emph{regular.}

\bigskip
Given $(\delbar_E, \varphi)$, the discriminant section $\Delta_{\varphi}$  is
\begin{eqnarray} \label{eq:discriminant}
\Delta_\varphi: C &\rightarrow& K_C^{n^2-n}\\ \nonumber
p &\mapsto& \prod_{1 \leq i<j \leq n} (\lambda_i(p) - \lambda_j(p))^2.
\end{eqnarray}
As shown in Figure \ref{fig:ram},
the associated spectral curve $\Sigma \overset{\pi}{\rightarrow} C$ is ramified at the zeros of the 
discriminant section $Z=\Delta^{-1}_\varphi(0) \subset C$.
\begin{figure}[ht]
 \begin{center}
  \includegraphics[height=.75in]{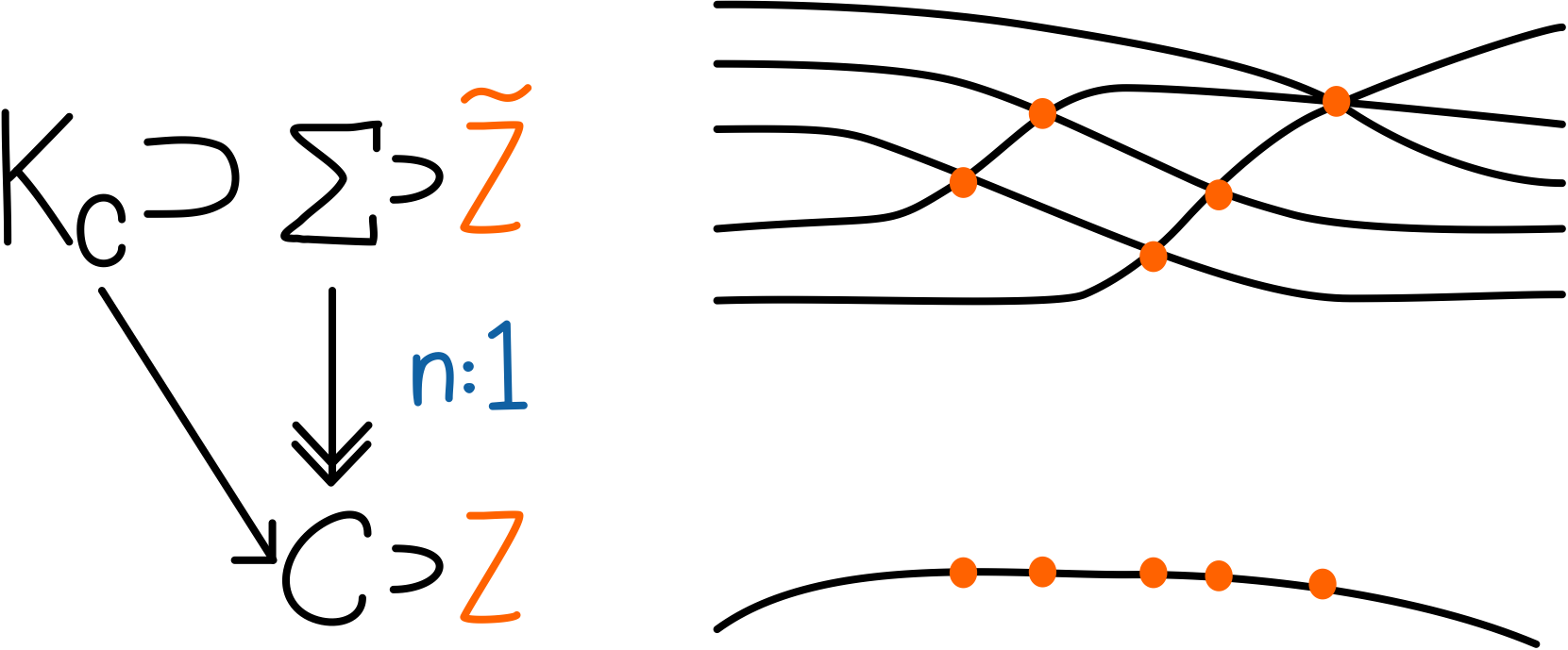}
  \caption{\label{fig:ram}The spectral cover $\Sigma$ is an $n:1$ cover of $C$, ramified at $Z$.}
 \end{center}
\end{figure}
Given $(\delbar_E, \varphi) \in \cM'$,
the map $\pi: \Sigma \rightarrow \C$ restricted to a neighborhood of a point $\widetilde{p} \in \widetilde{Z}$ looks like
\begin{eqnarray}
 \pi: \widetilde{\bD} &\rightarrow & \bD\\ \nonumber 
 w &\mapsto& w^K =z.
\end{eqnarray}
(This is indeed a smooth curve by the Jacobi criterion. The curve is the zero set of $f(w,z):=w^K-z$, and $\nabla f$ does not vanish at the point $(0,0)$.)
The point $\widetilde{p} \in \widetilde{Z}$ contributes a zero of order $K-1$ to $\Delta_\varphi$.  

\begin{rem}
 In the case where $n=2$,
a $SL(2,\C)$-Higgs bundle $(\delbar_E, \varphi)$ is in $\cM'$ if, and only
if, $\Delta_\varphi =-4\det \varphi$ has only simple zeros. Note that for $n>2$, the space of regular Higgs bundles is slightly larger than the space of Higgs bundles for which the discriminant section $\Delta_\varphi$ has only simple zeros.
\end{rem}

\subsection{Local model near a ramification point for a Higgs bundle \texorpdfstring{$(\delbar_E, \varphi) \in \cM'$}{}} \label{sec:Higgslocalmodel}

The next proposition gives a local model around ramification points $p \in Z$ for regular Higgs bundles.

\begin{figure}[ht]
\begin{centering}
\includegraphics[height=1.2in]{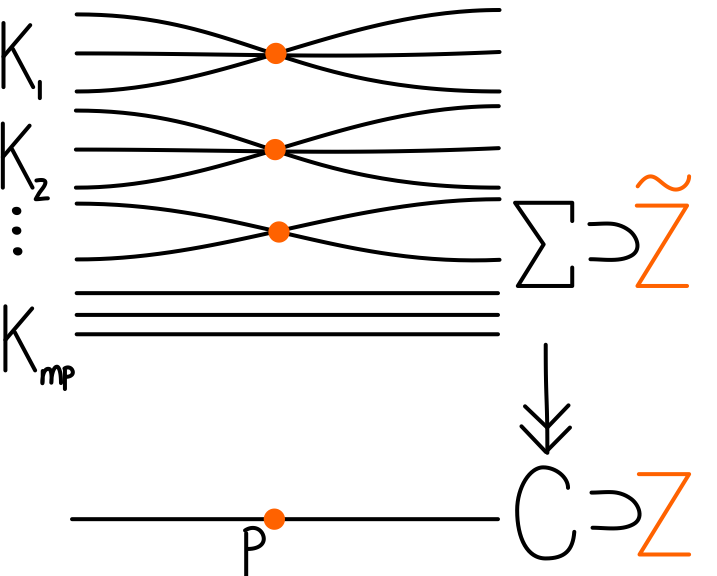}
\caption{\label{fig:model}
In the disk around $p \in Z$, we have $n=11$ and $K_1=K_2=3, K_3=2, K_4=K_5=K_6=1$.
}
\end{centering}
\end{figure}

\begin{prop}\label{prop:reghol}
(Local model for $(\delbar_E, \varphi)$ around ramification points)
Let $(\delbar_E, \varphi)$ be a polystable regular Higgs bundle.
Let $p \in Z \subset C$ be a ramification point.
Then, there are: a partition of $n$ as $n=K_1 +  \cdots + K_{m_p}$,
 local coordinates $z_1, \cdots, z_{m_p}$ centered at $p$,
and a local holomorphic trivialization of $E$ over a disk $\bD$ centered at $p$ such that 
  \begin{eqnarray} \label{eq:reghol}
   \delbar_E&=& \delbar\\ \nonumber
   \varphi &=& \bigoplus_{j=1}^{m_p} \left(\lambda_{(j)}\mathbf{1}_{K_j} + \begin{pmatrix}
               0 & 1 & &\\ & 0&  \ddots & \\
                &  & \ddots& 1\\
                z_j & & & 0
              \end{pmatrix}_{K_j \times K_j} \hspace{-.4in} \mathrm{d} z_j \hspace{.2in} \right).
  \end{eqnarray}
 Here, $\{\lambda_1, \cdots, \lambda_n\}$ are the eigenvalues of $\varphi$, and $\lambda_{(j)}$ is the average of the cluster of $K_j$ eigenvalues
  \begin{equation}
   \lambda_{(j)} = \sum_{k=s_{j-1} + 1}^{s_j} \lambda_k \qquad \mbox{where }s_i = \sum_{q=1}^i K_q.
  \end{equation}
\end{prop}

\begin{rem}
 For the $SL(2,\C)$ case, see \cite[Lemma 4.2]{MSWW14} which is considerably simpler and features an explicit gauge transformation. It is difficult to write such an explicit gauge transformation for arbitrary rank.
\end{rem}

\begin{proof}Take a disk $\bD$ centered at $p$ without additional ramification points.
Partition the eigenvalues by the value at $p$, and call these distinct values $\lambda_{(j)}(p)$. Let $K_j$ be the associated cluster size.
Because $\Sigma$ is smooth, there is exactly one sheet of $\Sigma$ going through $\lambda_{(j)}(p)$, and the spectral curve $\pi: \Sigma \rightarrow C$ through $\lambda_{(j)}(p)$ is locally given by 
\begin{eqnarray}
 \pi_j: \widetilde{\bD}_j &\rightarrow& \bD\\ \nonumber 
 w_j &\mapsto& w_j^{K_j}=z_j 
\end{eqnarray}
for some local holomorphic function $z_j$.
We can arrange that $z_j$ satisfies 
\begin{equation} 
\prod_{k=s_{j-1} + 1}^{s_j} (x + \lambda_{(j)} -\lambda_k ) = x^{K_j} - z_j \de z_j^{K_j}.
\end{equation}
(In the case $K=2$, this is equivalent to the standard argument (e.g. \cite[p. 216]{localholcoord}) showing that 
 that there is a local holomorphic
coordinate $z_i$ centered at $p$ such that 
\begin{equation}  \left(\lambda_{1}- \lambda_{2}\right)^2=4z_j dz_j^2.)\end{equation}

We can work locally with each cluster of size $K_j$, and for convenience
we may shift the eigenvalues so that $\lambda_{(j)}=0$; to avoid notational clutter, we drop all the indices $j$ related to cluster number, and number the eigenvalues $\lambda_1, \cdots, \lambda_{K}$.
The associated eigenvalues $\lambda_1, \cdots, \lambda_{K}$---or more precisely their pullbacks $\pi^*\lambda_i$---are single-valued
on the ramified $K:1$ local cover $\widetilde{\bD}$.
Order them so that $\lambda_j = \e^{\frac{2 \pi \I (j-1)}{K}} w \de (w^K)$.
Define
\begin{eqnarray}
 \sigma: \widetilde{\bD} &\rightarrow & \widetilde{\bD}\\ \nonumber 
 w &\mapsto& \e^{2\pi \I/K} w.
\end{eqnarray}
The cyclic group $\Z_{K}=\IP{\sigma}$ acts on $\widetilde{\bD}$, exchanging the sheets of $\pi: \widetilde{\bD} \to \bD$. Note that 
$\lambda_i =(\sigma^{i-1})^* \lambda_{1}$. 

Because the spectral cover $\Sigma$  is smooth, the associated
rank 1, locally-free, torsion-free sheaf $\cL \rightarrow \Sigma$
is actually a line bundle.
Thus, choose $s_1$ a smooth non-vanishing holomorphic section of the
eigenline associated to $\lambda_{1}$.
Define $s_{i} = (\sigma^{i-1})^* s_1$ and note that in the basis $s_{i}$ of $\pi^* \cE$, $\pi^* \varphi$ acts by multiplication by $\lambda_i$.
The basis elements $\{s_i\}$ do not descend from $\widetilde{\bD}$ to $\bD$, but the following basis elements satisfy $\sigma^*s_i'=s_i'$, and hence descend.
\begin{eqnarray} \label{eq:basiss}
 s'_1&=& \frac{1}{K} \sum_{i=1}^K s_i\\ \nonumber 
  s'_i &=& \frac{1}{K w^K (\de (w^K))^{K-(i-1)}} \sum_{k=1}^K \lambda_k^{K-(i-1)} s_k \qquad i=2, \cdots, K
\end{eqnarray}
Note that $s'_{i}$ is nonsingular and non-vanishing at $w=0$. 
In this basis,
\begin{equation}\label{eq:Higgs3}
\pi^* \varphi(s_1') = w^K \de (w^K) \varphi(s_K'), \qquad 
\pi^*\varphi(s_i') = \de (w^K) \varphi(s_{i-1}') \quad \mbox{for }i=2, \cdots, K.
\end{equation}
Define the basis $\e_i$ by $\pi^*\e_i = s_i'$. In this holomorphic basis the $K \times K$ block of $\varphi$ is 
\begin{equation}
\begin{pmatrix}
               0 & 1 & &\\ & 0&  \ddots & \\
                &  & \ddots& 1\\
                z & & & 0
              \end{pmatrix} \mathrm{d} z.
\end{equation}
If the average of the eigenvalues $\lambda_{(j)} \neq 0$, then we simply add $\lambda_{(j)} \mathbf{1}_K$, as claimed in \eqref{eq:reghol}.
Note that this holomorphic gauge is not unique since the section $s_1$ can be multiplied by any non-vanishing holomorphic function $f$.

In a block where $K=1$, the associated eigenvalue $\lambda$ is not ramified,
so we simply choose $\e$ to be a smooth section of the associated eigenline over the base $\bD$.
\end{proof}

\begin{rem}
 The sections $s'_i$ appearing in \eqref{eq:basiss} accomplish something slightly subtle. In the case $K=2$, Proposition \ref{prop:reghol} produces the following basis and sections:
 \begin{equation*}
  \varphi = \begin{pmatrix} 0 & 1 \\ z & 0 \end{pmatrix} \de z, \quad  s_1 = \pi^*\begin{pmatrix} 1\\ \sqrt{z} \end{pmatrix}, \lambda_1 =\pi^*( \sqrt{z} \de z),  \quad  s_2 = \pi^*\begin{pmatrix} 1\\ -\sqrt{z} \end{pmatrix}, \lambda_2 = \pi^*(-\sqrt{z} \de z).
 \end{equation*}
 In particular, note that $s_1$ and $s_2$ become linearly dependent at $w=0$. Despite this, the sections 
\begin{equation}
 s'_1 = \pi^*\begin{pmatrix} 1\\ 0  \end{pmatrix} \qquad s'_2 =\pi^* \begin{pmatrix} 0 \\ 1 \end{pmatrix}
\end{equation}
are linearly independent---even at $w=0$.
\end{rem}

\begin{rem} \label{rem:eps}
By shrinking $\bD$, we may assume that the disks around different points of $Z$
do not intersect. Shrinking
$\bD$ further, we may assume that the difference between the eigenvalues of $\varphi$ is
bounded below by some positive constant $\eps_\lambda>0$ on $C-\bigcup\limits_{p \in Z} \bD_p$.
By possibly taking a smaller $\eps_\lambda$, we may assume that \emph{on $\bD$}, the difference between the 
averaged-eigenvalues $\lambda_{(j)}$ 
are bounded below by $\eps_\lambda$.
By rescaling the Riemannian metric on $g_C$,
we may assume that each disk $\bD_p$ centered at $p$ has radius one.
\end{rem}

As shown in Figure \ref{fig:separate}, Proposition 2.5 gives a local
model only when the ramification points all lie above the same point.
Deforming this, we can also give a local model when the ramification points
lie above points that are nearby. This is the content of Corollary \ref{cor:goodgaugehol},
a direct corollary of the proof of Proposition \ref{prop:reghol}.
\begin{figure}[h!]
\begin{centering}
\includegraphics[height=1in]{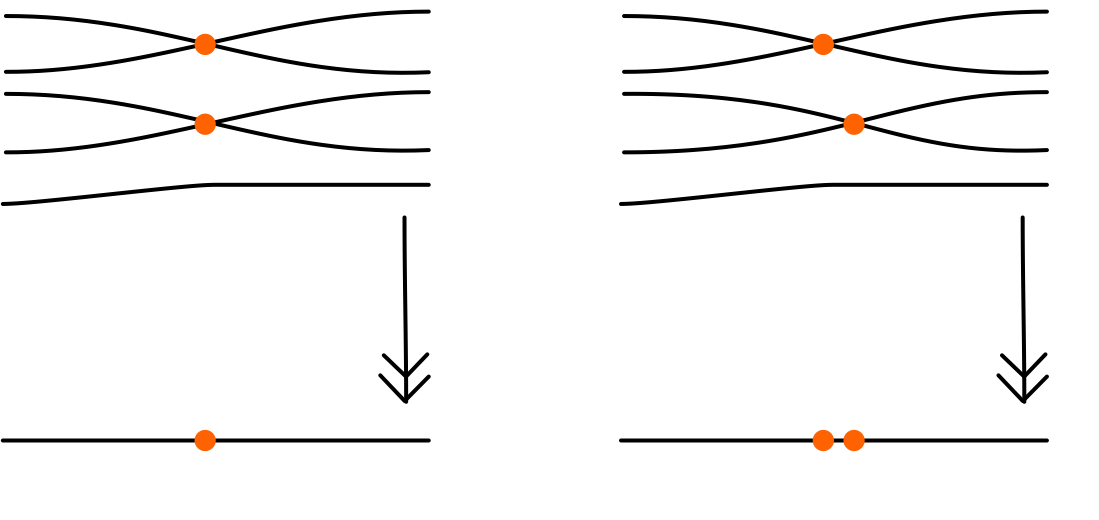}
\caption{\label{fig:separate}
Proposition \ref{prop:reghol} gives a local model
for local spectral covers corresponding to the left figure.  Corollary \ref{cor:goodgaugehol} 
gives a similar local model for the right figure.
}
\end{centering}
\end{figure}

\begin{cor}
\label{cor:goodgaugehol}
Let $(\delbar_E, \varphi)$ be a polystable regular Higgs bundle.
Let $\bD$ be an open neighborhood over which the spectral cover has $m_p$ connected components, each containing at most one point of $\widetilde{Z}$. 
Then, there are: a partition of $n$ as $n=K_1 +  \cdots + K_{m_p}$,
 local coordinates $z_1, \cdots, z_{m_p}$ with $z_i$ centered at $\pi(\widetilde{p}_i)$,
and a local holomorphic trivialization of $E$ over a disk $\bD$ centered at $p$ such that 
  \begin{eqnarray}
   \delbar_E&=& \delbar\\ \nonumber
   \varphi &=& \bigoplus_{j=1}^{m_p} \left(\lambda_{(j)}\mathbf{1}_{K_j} + \begin{pmatrix}
               0 & 1 & &\\ & 0&  \ddots & \\
                &  & \ddots& 1\\
                z_j & & & 0
              \end{pmatrix}_{K_j \times K_j} \hspace{-.4in} \mathrm{d} z_j \hspace{.2in} \right).
  \end{eqnarray}
\end{cor}

\subsection{Stratification of \texorpdfstring{$\cM'$}{M'}} \label{sec:stratification}

Given a Higgs bundle $(\delbar_E, \varphi) \in \cM'$, we get a collection 
of partitions $n=K_{1,p} + K_{2,p} + \cdots + K_{m_p,p}$ labeled
by $p \in Z$.  For $K=2, 3, \ldots, n$ define
\begin{equation}N_K=\# \{(p,i) \in Z \times \mathbb{N}:   K_{p,i} =K\}.
\end{equation}
(Since $\Delta_\varphi$ is a section of $K_C^{n^2-n}$,  
\begin{equation}
 \sum_{K=2}^n (K-1) N_K = 2(n^2-n)(g-1).)
\end{equation}
This gives us a map onto a discrete space:
\begin{eqnarray}
 \Xi: \cM' &\rightarrow& \mathbb{N}^{n-1}\\ \nonumber
 (\delbar_E, \varphi) &\mapsto& (N_2, \cdots, N_n).
\end{eqnarray}
The map $\Xi$ gives us a stratification of $\cM'$.

\section{Limiting configurations}\label{sec:limitingconfigurations}

One of the salient properties of the limiting metric $h_\infty=\lim_{t \to \infty} h_t$ is that it
solves 
the ``decoupled $SU(n)$-Hitchin's equations'' by \cite[Theorem 2.7]{Mochizuki:2015aa}.
\begin{defn} \label{defn:limitingconfiguration}
Given a polystable Higgs bundle $(\delbar_E, \varphi) \in \cM$, a hermitian metric $h$ solves the \emph{decoupled $SU(n)$-Hitchin's equations} if
\begin{equation}
 [\varphi, \varphi^{\dagger_{h}}]=0, \qquad F_{D(\delbar_E, h)}^\perp =0.
\end{equation}
and $\det h=h_{\Det E}$.
\end{defn}
Fix a polystable regular Higgs field $(\delbar_E, \varphi)$. 
In this section, we construct a metric $h_\ell$ solving the
decoupled $SU(n)$-Hitchin's equations. 
It is worth emphasizing that there are many solutions of the decoupled $SU(n)$-Hitchin's equations, and each of these solutions 
depends on a choice of parabolic weights (Remark \ref{rem:weights}).
We make the ``correct'' choice of parabolic weights in our construction, though
this is only justified \emph{a posteriori} in Corollary \ref{cor:main} when we prove that
$h_\ell=h_\infty$. The subscript $\ell$ is used for ``limiting.''

\subsection{Construction of limiting metrics}\label{sec:limitingconstruction}

Given a polystable regular Higgs field $(\delbar_E, \varphi)$, Construction \ref{construction:limiting}
produces a singular hermitian metric $h_\ell$, 
unique up to rescaling by a constant.
This metric
arises as the pushforward of the Hermitian-Einstein metric on the associated spectral
line bundle $\cL \rightarrow \Sigma$ equipped with a specific parabolic structure.  
By Proposition \ref{prop:limitingconstruction},
the triple $(\delbar, \varphi, h_\ell)$ 
solves the decoupled $SU(n)$-Hitchin's equations---possible after some constant rescaling of $h_\ell$. As mentioned above, in Corollary \ref{cor:main}, we will prove that 
$\lim_{t \to \infty} h_t = h_\ell$. Thus, we call this particular
triple which solves the decoupled $SU(n)$-Hitchin's equations a \emph{limiting configuration}.

\begin{figure}[h!]
\begin{centering}
\includegraphics[height=1.2in]{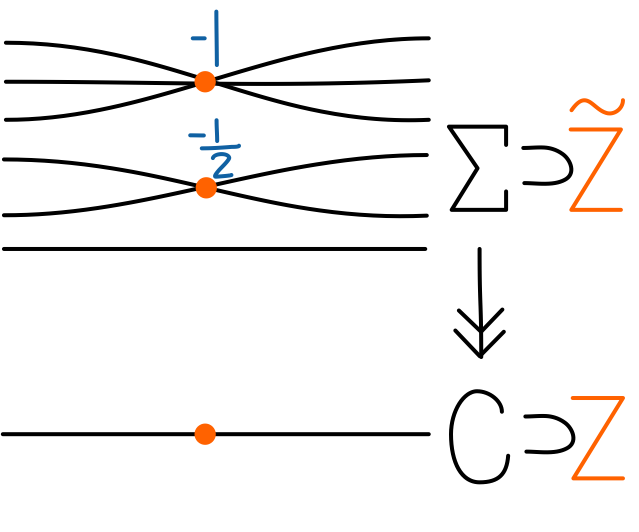}\\
\vspace{-.2in}
\caption{\label{fig:hl}
At a point $\widetilde{p} \in \widetilde{Z}$ where the spectral cover is locally $K:1$, put parabolic weight $\frac{1-K}{2}$.
}
\end{centering}
\end{figure}

\begin{construction}\label{construction:limiting}
\noindent Given $(\delbar_E,  \varphi)$ a regular polystable Higgs bundle,
let $\cL \rightarrow \Sigma$ be the associated
spectral data.

\begin{itemize}
\item Equip the holomorphic line bundle $\cL \rightarrow \Sigma$ with parabolic structure: 
\emph{At each point $\widetilde{p}_j\in \widetilde{Z}$ add the parabolic weight $\frac{1-K_j}{2}$ (as shown in Figure \ref{fig:hl}). 
}
\medskip
\item Equip the parabolic line bundle $\cL \rightarrow \Sigma$ with a hermitian structure: 
\emph{For parabolic line bundles---such as $\cL$---
there is a Hermitian-Einstein metric adapted to the 
parabolic structure\cite{simpsonnoncompact,biquardparabolic}\footnote{Technically,
the result in \cite{biquardparabolic} is only for parabolic line bundles of parabolic degree $0$, however, it is straightforward to extend the results to arbitrary degree.}.
The Hermitian-Einstein metric solves
 \begin{equation}\label{eq:parabolichermitianeinstein}
 F_{\cL} = -2 \pi \sqrt{-1} \frac{\mathrm{pdeg} \;\cL }{\rk\; \cL } 
  \frac{ \pi^*\omega_C}{\mathrm{vol}_{\pi^*{g_C}}(\Sigma)}\Id_\cL .
\end{equation}
and is unique up to rescaling by a constant.
}
\medskip
\item Define $h_\ell$ on $E|_{C-Z}$ 
from the orthogonal push-forward
of the Hermitian-Einstein metric $h_{\cL}$ on $\cL \rightarrow \Sigma$.
I.e. decompose $E$ into eigenspaces of $\varphi$;
these eigenspaces are orthogonal with respect to $h_\ell$; on each eigenspace
$h_\ell$ agrees with the metric induced by $h_{\cL}$.
\end{itemize} 
\end{construction}

\begin{prop}\label{prop:limitingconstruction}
Given a polystable regular Higgs bundle,
Construction \ref{construction:limiting} produces a unique
hermitian metric $h_\ell$ solving the $SU(n)$-decoupled Hitchin's equations.
\end{prop}

\begin{proof} 
Construction \ref{construction:limiting} determines
a hermitian metric $h_\ell$ on $E_{C-Z}$ up 
to rescaling by a constant.  Any such metric $h_\ell$ solves the decoupled Hitchin's equations.
Since $\varphi$ and $h_\ell$
are diagonal in the basis of eigenbundles on $\varphi$,
$\left[\varphi, \varphi^{\dagger_{h_\ell}}\right]=0$.

\smallskip

\noindent\textsc{Claim:} The parabolic degree of $\cL$ is equal to the degree of $E$.
\\
\emph{Proof:} $\triangleright$
 The statement 
$\mathrm{pdeg} \mathcal{L}
=\mathrm{deg} E$ holds because of the choice of parabolic weights.
A cluster of size $K$ contributes a zero of order $K-1$ to $\Delta_\varphi$; at such a point $\widetilde{p} \in \widetilde{Z}$, we assigned the parabolic weight $\frac{1-K}{2}$.
Since $\Delta_\varphi$ has $2(n^2-n)(g-1)$ zeros (counted with multiplicity),
the sum of all parabolic weights is $-\frac{1}{2}  \cdot 2(n^2-n)(g-1)$.
Consequently, 
\begin{eqnarray}
\mathrm{pdeg} \;\cL &=& \deg \cL + \sum_{\tilde{p} \in \tilde{Z}} \alpha_p \\ \nonumber
&=& \left(\deg E+ (n^2-n)(g-1) \right) + \left(-\frac{1}{2}\right)2(n^2-n)(g-1) \\ \nonumber
&=& \mathrm{deg}\; E.  \hspace{4.5in} \triangleleft
\end{eqnarray}

\smallskip

The condition
$F^\perp_{D(\delbar_E, h_\ell)}=0$ holds  because $\mathrm{pdeg} \;\cL=\mathrm{deg}\; E$.
The induced metric $\det(h_\ell)$
is a Hermitian-Einstein metric on $\Det E$,
consequently it is a constant multiple of the fixed Hermitian-Einstein metric $h_{\Det E}$.
Rescale $h_\ell$ by a constant so that these two Hermitian-Einstein metrics agree.
 \end{proof}

\begin{rem}\label{rem:weights}
Note that in the proof we did not use the individual values of the parabolic weight.
We only used the fact that the sum of all parabolic weights was $-(n^2-n)(g-1)$.
In Construction \ref{construction:limiting}, we could take any collection of parabolic weights summing to $-(n^2-n)(g-1)$ and produce 
a hermitian metric solving the decoupled $SU(n)$-Hitchin's equations.
However, this hermitian metric agrees with $h_\infty$ only for our choice of parabolic weights (Corollary \ref{cor:main}).
\end{rem}

\subsection{Local model near a ramification point for a limiting configuration in \texorpdfstring{$\cM'$}{M}} \label{sec:gen2}

The next proposition gives a local model
for the limiting configuration
in Proposition \ref{prop:limitingconstruction}.

\begin{prop}\label{prop:reggoodgaugehinfty}
There is a holomorphic gauge satisfying the conditions of Proposition \ref{prop:reghol} in which 
  \begin{eqnarray} \label{eq:reggoodgaugehinfty}
    h_\ell &=& \bigoplus_{j=1}^{m_p}  \begin{pmatrix}
     |z_j|^{-2 \alpha_{K_j,1}}& & \\
      & \ddots & \\
      & & |z_j|^{-2 \alpha_{K_j,K_j}}  \end{pmatrix}_{K_j \times K_j}.
  \end{eqnarray}
Here, the constants $\alpha_{K,i}$ are 
\begin{equation} \label{eq:regalpha}
 \alpha_{K,i} = \frac{2i-(K+1)}{2K}.
\end{equation}
\end{prop}

\begin{proof}
First, assume $\deg E=0$.
We can work locally with each cluster of size $K_j$. As in the proof of Proposition \ref{prop:reghol}, we assume that we are working with the first cluster and drop all subscripts relating to the cluster index. The key idea is that 
we use up the remaining gauge freedom in Proposition \ref{prop:reghol} by multiplying the section $s_1$
by a non-vanishing holomorphic function in order
to arrange that 
\begin{equation}
 h_\cL\left(s_1, s_1 \right) =\sqrt{K}|w|^{1-K}.
\end{equation}

Let $h_{\cL}$ be the Hermitian-Einstein metric on 
$\cL \rightarrow \Sigma$ which is adapted
to the hermitian metric. 
There are two consequences of this. First, because
$h_{\cL}$ is adapted
to the parabolic structure on $\cL \rightarrow \Sigma$
at $\tilde{p}_i$,
$h_{\cL}(s_{1}, s_{1}) \sim 
|w|^{-2 \cdot \frac{1-K}{2} }=|w|^{K-1}$.
Secondly, because $h_\cL$ is Hermitian-Einstein and $\log h_{\cL}(s_{1}, s_{1})$ is harmonic. 
Any harmonic function on the punctured disk $\widetilde{\bD}^\times$ can be written  $\mathrm{Re}(f(w))+c \log (|w|)$ where $f(w)$ is holomorphic on $\widetilde{\bD}^\times$ 
 and $c$ is some constant; hence
\begin{equation}
 \log h_\cL(s_{1}, s_{1}) = \mathrm{Re}(f(w)) + (K-1) \log (|w|).
\end{equation}
The function $f$ is bounded because $h_\cL$ is adapted, hence it extends to a holomorphic function on $\bD$.
We replace  $s_{1}$ with the section  $K^{1/2}\e^{-\frac{f}{2}} s_{1}$ 
that satisfies
\begin{equation}
 \log h_\cL\left(K^{1/2}\e^{-\frac{f}{2}} s_1, K^{1/2}\e^{-\frac{f}{2}} s_1 \right) = \log \left(K \e^{-\mathrm{Re}(f)} h_\cL(s_1, s_1) \right)=(K-1) \log(|w|) + \log K. 
\end{equation}
Consequently, $h_\cL(s_i, s_i) = K |w|^{K-1}$ for $i=1, \cdots K$.
Then, we see that
\begin{equation}
 \pi^*h_\ell(s_i, s_k) = \begin{cases} 0  &\mbox{if } i \neq k \\
                         |w|^{K+1-2i} &\mbox{if } i =k
                        \end{cases}.
\end{equation}
Hence in the basis $\{e_i\}$, the hermitian metric is as claimed.

Note that when $K=1$, the associated eigenvalue $\lambda$ is not ramified.
Consequently, the associated section $e$ satisfies 
$\log h_\ell(e, e)= \mathrm{Re}(f(z))$ for $f(z)$ harmonic on the disk $\bD$---rather than its cover.
Thus, by replacing $e$ with the section  $\e^{-\frac{f}{2}} e$,
we see that $h_\ell(\e^{-\frac{f}{2}} e,\e^{-\frac{f}{2}} e)=1$, as desired.

Note that $\det h_\ell=1$ because $h_\ell$ is block diagonal and the determinant of each $K \times K$ block is $1$.

\medskip

If $\deg E \neq 0$, then we simply note that $\log h_\cL(s_1, s_1)$ minus some multiple of the K\"ahler potential
is harmonic, and repeat the argument above.
\end{proof}

\section{A family of approximate solutions}\label{sec:approximatesolutions}

Fix a regular polystable Higgs bundle $(\delbar_E, \varphi)$.
Consider the $\R^{+}_t$-family of Higgs bundles 
$(\delbar_E, t \varphi)$. 
Ultimately, 
we seek to describe the corresponding family of harmonic
metric $h_t$ for large values of $t$.
In this section,
we construct a $\R^{+}$-family of approximate solutions $h_t^\app$ 
by desingularizing the limiting configuration $h_\ell$ in $\S$\ref{sec:limitingconfigurations}.
As shown in Figure \ref{fig:globalgluing}, the metric $h_\ell$ is singular at $p \in Z$, 
 so we glue
in smooth solutions 
of Hitchin's equations on the disks $\bD$ around each ramification point $p \in Z$. These smooth models are described in \S\ref{subsec:modelsolutions}.
Because these smooth models are defined on disks in $\C$ with its usual flat metric,
we take a conformal metric $g_C'$ on $C$ which is flat in each disk $\bD$ (Remark \ref{rem:metric}).
The approximate solutions $h_t^\app$ are defined in \S\ref{sec:defnapprox}.
\begin{figure}
\begin{centering} 
\includegraphics[height=1in]{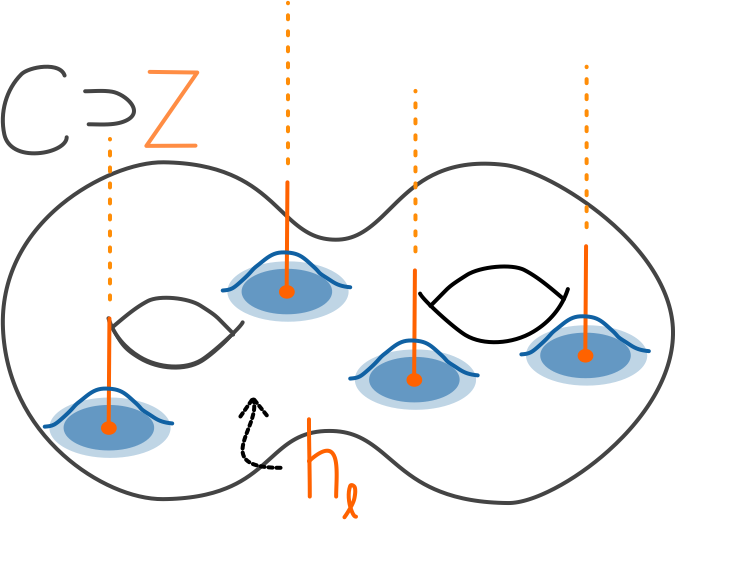}\\
\vspace{-.2in}
\caption{\label{fig:globalgluing} The curvature $F_{D(\delbar, h_\ell)}^\perp$ is concentrated at $p \in Z$, illustrated by orange spikes.  Approximate solutions $h_t^\app$ are constructed by desingularizing $h_\ell$ by gluing in smooth model solutions (shown in blue).
}
\end{centering}
\end{figure}

\subsection{Model solutions}\label{sec:modelsolutions}
For each cluster rank $K$, we describe the necessary family of model solutions of rank parameterized by $t \in \R^+$. All of these model solutions are on $\C$ with its flat metric.
We begin by reviewing the $K=2$ family of model solution featured in \cite{MSWW14} in \S\ref{sec:2model}
before turning to the higher rank versions in \S\ref{sec:MK1}. We conclude by describing the model solutions
for the regular Higgs bundle $(\delbar_E, t \varphi)$ on the disk $\bD$ in \S\ref{subsec:modelsolutions}.

\subsubsection{The $K=2$ family of model solution}\label{sec:2model}
The following family of model solutions is featured in \cite{MSWW14}.

\begin{defn} \label{defn:model2}
The \emph{$SU(2)$ $t$-model solution}
 is 
 \begin{eqnarray} \label{eq:model2}
   \delbar^{\2}_E&=& \delbar \\ \nonumber
   \varphi^{\2} &=& \begin{pmatrix} 0 & 1 \\ 
 z & 0 \end{pmatrix} \de z\\\nonumber
   h_t^{\2, \fid} &=& \begin{pmatrix} |z|^{1/2}e^{u_t(|z|)} &\\& |z|^{-1/2}e^{-u_t(|z|)}\end{pmatrix}
  \end{eqnarray}
  where $u_t:\R^{+} \rightarrow \R$ is solution of
  \begin{equation} \label{eq:Painleve}
  \left(\frac{\de^2}{\de|z|^2} + \frac{1}{|z|} \frac{\de}{\de |z|} \right) u_t=
  8 t^2 |z| \sinh(2 u_t). \end{equation}
with asymptotics \\
\begin{tabular}{ l l l }\label{eq:asymptotics}
 \hspace{2in} &$u_t(|z|) \sim  \frac{1}{\pi} K_0(\frac{8t}{3}|z|^{\frac{3}{2}})$ & as $|z|\rightarrow \infty$ \\
 &$u_t(|z|) \sim - \frac{1}{2} \log (|z|)$ 
 & as $|z| \rightarrow 0$.
\end{tabular}
\end{defn}
\begin{rem}
 The $u_t$ are related by
 \begin{equation} \label{eq:rescaleu}
  u_t=\rho_t^*u_1 \qquad \rho_t(z) = t^{2/3} z.
 \end{equation}
\end{rem}

\begin{rem}\label{rem:Painleve}Note that $h_t^{\2, \fid}$ has a chance of being smooth at $|z|=0$ because of
the coefficient of $\log(|z|)$ appearing in the expansion around $|z|=0$.
Mazzeo-Swoboda-Weiss-Witt prove that it is smooth in \cite[Corollary 3.4]{MSWW14}.
Moreover, note that the pointwise limit $\lim_{t \to \infty} h_t^{\2, \fid}$
is $\mathrm{diag}(|z|^{1/2}, |z|^{-1/2})$. This is the $2 \times 2$ 
block appearing in the limiting metric $h_\ell$ in \eqref{eq:reggoodgaugehinfty}.
\end{rem}
\begin{rem}
In unitary gauge (see the discussion at the end of \S\ref{sec:fixeddata}), the $SU(2)$ t-model solution is written
\begin{eqnarray}
A_t^{(2),\fid}&=&
d+ \left(\frac{1}{8} +\frac{|z|}{4} \frac{du_t}{d|z|} \right) \begin{pmatrix} 1 & \\  & -1 \end{pmatrix} \left( \frac{\de z}{z} - \frac{\de \zbar}{\zbar} \right)\\ \nonumber
\Phi^{(2),\fid}_t&=&
\begin{pmatrix} 0 & |z|^{1/2}e^{u_t(|z|)} \\
 \frac{z}{|z|^{1/2}}e^{-u_t(|z|)} & 0 \end{pmatrix}dz
\end{eqnarray}
In \cite{MSWW14}, as well as in \cite{GMNhitchin},
the model solution takes this shape.
\end{rem}

\subsubsection{The rank $K$ family of model solutions from \texorpdfstring{$\cM_{K,1}$}{MK1}} \label{sec:MK1}
For any rank $K$, there is a single $SU(K)$ model solution on $\C$ with its flat metric which generalizes the
the $SU(2)$ $(t=1)$-model solution in \S\ref{sec:2model}.
For each $K$, there is a one-point moduli space $\cM_{K,1}$ of solutions of the $SU(K)$-Hitchin's equations
on $\CP^1$ with an irregular singularity at $\{\infty\}$ such that the eigenvalues of the Higgs field are $\lambda_r = \e^{2 \pi i r/K} z^{1/K} \de z$. 
The point $[(\delbar_E, \varphi, h)]$ is fixed by a $U(1)$-action, consequently the solution of Hitchin's equations can be written down
relatively explicitly \cite{FredricksonNeitzke}:

\begin{prop}\cite[Proposition 3.9 \& Lemma 3.13]{FredricksonNeitzke} \label{prop:MK1model}
The one-point point moduli space $\cM_{K,1}=[(\delbar_E, \varphi, h^{(K), \model})]$ where 
\begin{equation} \label{eq:MK1model}
\delbar_E = \delbar
  , \qquad \varphi = \begin{pmatrix}
               0 & 1 & &\\ & 0&  \ddots & \\
                &  & \ddots& 1\\
                z & & & 0
              \end{pmatrix} \mathrm{d} z, \qquad
                h^{(K), \mathrm{mod}} = \begin{pmatrix}
     |z|^{-2 \alpha_{K,1}} \mathrm{e}^{u_{K,1}} & & \\
      & \ddots & \\
      & & |z|^{-2 \alpha_{K,K}} \mathrm{e}^{u_{K,K}} \end{pmatrix}.\end{equation}
The constants $\alpha_i$ are 
\begin{equation}
 \alpha_{K,i} = \frac{2i-(K+1)}{2K}.
\end{equation}
The real-valued functions $u_{K,i}(z) =u_{K,i}(|z|)$ satisfy the symmetry $u_{K,i} = -u_{K, K+1-i}$ 
and solve
 \begin{equation}\label{eq:ODE}
\frac{1}{4} \left(\frac{\mathrm{d}^2}{\mathrm{d} |z|^2}+ \frac{1}{|z|} \frac{\mathrm{d}}{\mathrm{d}|z|} \right) u_{K,i} =
|z|^{\frac{2}{K}} \left(\mathrm{e}^{u_{K,i}-u_{K,i+1}} - \mathrm{e}^{u_{K,i-1}-u_{K,i}} \right)
\end{equation}
with the following boundary conditions:
\begin{itemize}
\item The function $u_{K,i}$ decays to $0$  as $\vert z \vert \rightarrow \infty$.
\item Near $0$, $u_{K,i} \sim 2 \alpha_{K,i}\log \vert z \vert$.
\end{itemize}
Letting  $\mathbf{u}(|z|) = \left(u_{K,1}(|z|), \dots, u_{K,K}(|z|)\right)$, 
the function $\norm{\mathbf{u}(|z|)}^2$ is decreasing and exhibits exponential decay at $\infty$. 
More precisely, for $\eps>0$, take $R_\eps>0$ such that $\|\mathbf{u}(R_\eps)\| < \eps$.
Then, there is a constant $c>0$ (depending explicitly on $\eps$ and $K$) such that
\begin{equation}\label{eq:85b}
\|\mathbf{u}\|^2 (\rho)  \leq \eps^2 \frac{ K_0(c \zeta(\rho))}{K_0(c \zeta(R_\eps))} \qquad \mbox{for $\rho>R_\eps$},
\end{equation}
where $K_0$ is the modified Bessel function of first kind and $\zeta(|z|)=\frac{2K}{K+1}|z|^{\frac{(K+1)}{K}}$.
\end{prop}

\begin{rem} The bound in \eqref{eq:85b} is not sharp. 
The constant $c=(2 C_\eps C_K)^{-1/2}$ where $C_\eps>1$ with $\lim_{\eps \searrow 0} C_\eps=1$; the first few values $C_K$ are $C_2=4$, $C_3=3$, $C_4=2$, $C_5=\frac{5 - \sqrt{5}}{2}$. In the case where $K=2$, $u_1=-u_2 \sim K_0( \frac{8}{3} \rho^{3/2})$; consequently the constant $\frac{4}{3}c$ at best approaches $\frac{\sqrt{8}}{3}$---considerably worse than the optimal constant $\frac{16}{3}$.
\end{rem}

\begin{rem}
With the change of variables above, $u_{K,i}(\zeta)$ solve the system of equations
\begin{equation}\label{eq:toda}
\left(\frac{\de^2}{\de \zeta^2}+ \frac{1}{\zeta} \frac{\de}{\de\zeta} \right) u_{K,i} =
\e^{u_{K,i}-u_{K,i+1}} - \e^{u_{K,i-1}-u_{K,i}}.
\end{equation}
This is the radial
version of the coupled system of PDE known as ``2d cyclic affine Toda lattice with opposite sign.'' 
Because of the symmetry, this is actually a coupled system of $\lfloor \frac{K-1}{2} \rfloor$ ODEs.

The solution of Hitchin's equations in Proposition \eqref{prop:MK1model} appears earlier in the literature, where it is called a solution of the ``$tt^*$-Toda equations.'' The $tt^*$-Toda equations are a special case of the
$tt^*$-equations which were introduced by Cecotti and Vafa to describe
certain deformations of supersymmetric quantum field theories \cite{Cecotti:1991me, CecottiVafaExact}.
(Not every solution of the $tt^*$-equations is a solution of Hitchin's equations
on a Riemann surface, and conversely,  not every solution of Hitchin's equations gives a solution of the $tt^*$-equations. However, these coincide here roughly because
$\cM_{K,1}$ is a one-point moduli space fixed by a circle action and a real involution.)
These particular solutions were also studied in \cite{GuestLin, GuestLinIsomonodromy, MochizukiToda}. 
\end{rem}

\medskip

We now introduce the parameter $t \in \R^+$.
Define rescaled functions 
\begin{equation} \label{eq:regrescale}
u_{K,i, t} =\rho_{K,t}^* u_{K,i}, \qquad \mbox{where } \rho_{K,t}: r \rightarrow t^{\frac{K}{K+1}} r.
\end{equation}
The following triple $(\delbar_E, t \varphi, h_t^{(K), \mathrm{mod}})$ solves Hitchin's equations:
\begin{equation} \label{eq:MK1modelt}
\delbar_E = \delbar
  , \qquad t \varphi = t\begin{pmatrix}
               0 & 1 & &\\ & 0&  \ddots & \\
                &  & \ddots& 1\\
                z & & & 0
              \end{pmatrix} \mathrm{d} z, \qquad
                h_t^{(K), \mathrm{mod}} = \begin{pmatrix}
     |z|^{-2 \alpha_1} \mathrm{e}^{u_{K,1,t}} & & \\
      & \ddots & \\
      & & |z|^{-2 \alpha_K} \mathrm{e}^{u_{K,K,t}} \end{pmatrix}.\end{equation}

  \subsubsection{Family of model solutions for $(\delbar_E, t \varphi)$} \label{subsec:modelsolutions}
We will use the the following family of model solutions to desingularize $h_\ell$.
  
\begin{defn} 
Let $(\delbar_E, \varphi,h_\ell)$ be as in \eqref{eq:reghol} \& \eqref{eq:reggoodgaugehinfty}.
Define a hermitian metric
   \begin{eqnarray}
    h_t^{\mathrm{mod}} &=& \bigoplus_{j=1}^{m_p}  \begin{pmatrix}
     |z_j|^{-2 \alpha_{K_j,1}} \, \mathrm{e}^{u_{K_j,1,t}(|z_j|)} & & \\
      & \ddots & \\
      & & |z_j|^{-2 \alpha_{K_j,K_j}} \, \mathrm{e}^{u_{K_j, K_j, t}(|z_j|)} \end{pmatrix}_{K_j \times K_j}.
  \end{eqnarray}
Call the $t$-family $(\delbar_E, t \varphi, h_t^{\mathrm{mod}})$ the \emph{family of model solutions} of Hitchin's equations.
\end{defn}

\begin{rem}
In unitary gauge, this is
 \begin{eqnarray} \label{eq:regmodunitary}
 A^{\mathrm{mod}}_t &=&  \bigoplus_{j=1}^{m_p}  \begin{pmatrix}
   - \frac{\alpha_{K_j,1}}{2} + \frac{|z_j|}{4} \frac{\de u_{K_j,1,t}}{\de |z_j|} & & \\
      & \ddots & \\
      & &  -\frac{\alpha_{K_j,K_j}}{2} +\frac{|z_j|}{4} \frac{\de u_{K_j,K_j,t}}{\de |z_j|} \end{pmatrix}_{K_j \times K_j} \hspace{-.4in}\left( \frac{\de z_j}{z_j} - \frac{\de \zbar_j}{\de \zbar_j} \right)\\ \nonumber 
\Phi^{\mathrm{mod}}_t &=& \bigoplus_{j=1}^{m_p} \left(\lambda_{(j)}\mathbf{1}_{K_j} + \begin{pmatrix}
               0 & |z_j|^{\frac{1}{K_j}}\e^{\frac{u_{K_j,1, t} -u_{K_j,2,t}}{2}} & &\\ & 0&  \ddots & \\
                &  & \ddots& |z_j|^{\frac{1}{K_j}} \e^{\frac{u_{K_j,K_j-1, t} -u_{K_j,K_j,t}}{2}}\\
                z_j |z_j|^{-\frac{K_j-1}{K_j}} \e^{\frac{u_{K_j,K_j, t} -u_{K_j,1,t}}{2}}& & & 0
              \end{pmatrix}  \mathrm{d} z_j \right).\hspace{-1.0in}
\end{eqnarray}
Note that if $K_j=1$, the $1 \times 1$ block in $A^{\mathrm{mod}}_t$ is $(0)$ and the block in $\Phi_t^{\mathrm{mod}}$ is the eigenvalue $(\lambda_{(j)})$.
 \end{rem}

\subsection{Description of approximate solutions}\label{sec:defnapprox}
The following non-linear operator measures the
the failure of $(\delbar_E, \varphi, h)$ to be a solution of Hitchin's equations:
\begin{equation}\label{eq:F1}\mathbf{F}(\delbar_E, \varphi, h)
:=H^{1/2} \left( F_{D(\delbar_E, h)}^\perp + [ \varphi ,  \varphi^{\dagger_h} ] \right) H^{-1/2}.
\end{equation}
Observe that we conjugate by the $\End(E)$-valued section $H^{1/2}$ (discussed at the end of \S\ref{sec:fixeddata}) which satisfies $h(v, w)=h_0(H^{1/2} v, H^{1/2}w)$. By doing this, the output $\mathbf{F}(\delbar_E, \varphi, h)$ is an $h_0$-unitary section of $\Omega^{1,1}(C, \mathfrak{su}(E))$.
(Equivalently in the unitary formulation of Hitchin's equations, this operator $\mathbf{F}$ is equal to 
 $\mathbf{F}(\de_A, \Phi) = F^\perp_A + [\Phi, \Phi^{\dagger_{h_0}}].$)

\begin{defprop}\label{defpropreg}
Choose a smooth cutoff function 
$\chi: [0,\infty) \rightarrow [0,1]$ such that
\begin{eqnarray}
\chi\Big|_{[0, \frac{1}{2}]}=1 &\mbox{and}& \chi\Big|_{[1, \infty)}=0. 
\end{eqnarray}
On $\bD_p$, in the local gauge of Proposition \ref{prop:reggoodgaugehinfty}, define $h_t^{\app}$ by
  \begin{eqnarray} \label{eq:regHapp}
    h_t^\app &=& \bigoplus_{j=1}^{m_p}  \begin{pmatrix}
     |z_j|^{-2 \alpha_{K_j,1}} \, \mathrm{e}^{\chi(|z_j|) u_{K_j,1,t}(|z_j|)} & & \\
      & \ddots & \\
      & & |z_j|^{-2 \alpha_{K_j,K_j}} \, \mathrm{e}^{\chi(|z_j|) u_{K_j,K_j,t}(|z_j|)} \end{pmatrix}_{K_j \times K_j}.
  \end{eqnarray}
On  $C^\rext=C- \overline{ \cup_p \bD_p}$, define $h_t^{\app}=h_\ell$.

For $t_0>0$ sufficiently large, 
there exists positive constants $c, \delta$ such that for $t>t_0$
 \begin{equation} \Big\|  \mathbf{F}(\delbar_E, t\varphi, h_t^{\app}) \Big\| 
 _{L^2(C)} \leq c e^{-\delta t},\end{equation}
for $\mathbf{F}$ defined in \eqref{eq:F1}.
 Because of the exponential decay in $t$, call the family $\{(\delbar_E, t \varphi, h_t^\app)\}_{t>t_0}$ a 
 \emph{family of approximate solutions}.
 \end{defprop}

  \begin{rem} \label{rem:regunitarygauge}
  In unitary gauge, the family $(A_t^\app, \Phi_t^\app)$ is given by inserting the cutoff function $\chi$ into the expressions $(A_t^{\mathrm{mod}}, \Phi_t^{\mathrm{mod}})$ in \eqref{eq:regmodunitary}
 \end{rem}

\begin{proof}[Proof of Proposition \ref{defpropreg}]
On $C^\rext$, $h_t^\app=h_\ell$.
Because $(\delbar_E, \varphi, h_\ell)$ solves the decoupled Hitchin's equations, 
$\mathbf{F}(\delbar_E, t\varphi, h_t^{\app})$ vanishes on $C^\rext$.
The $L^2(C)$-norm is simply the sum of the $L^2(\mathbb{D}_p)$-norms
of each of the individual $K_j \times K_j$ blocks.
Dropping indices, the relevant $K \times K$ piece is
\begin{equation}
 \delbar_E =\delbar \quad t \varphi = t \begin{pmatrix}
               \lambda & 1 & &\\ & \lambda &  \ddots & \\
                &   & \ddots& 1\\
                z & & & \lambda
              \end{pmatrix} \mathrm{d} z, \quad
                h_t^\app = \begin{pmatrix}
     |z|^{-2 \alpha_{K,1}} \mathrm{e}^{\chi u_{K,1,t}} & & \\
      & \ddots & \\
      & & |z|^{-2 \alpha_{K,K}} \mathrm{e}^{\chi u_{K,K,t}} \end{pmatrix}.\end{equation}
       On the $K \times K$ block, the value of $F_{D(\delbar_E, \varphi)}+t^2[\varphi, \varphi^{\dagger_{h_t^\app}}]=0$ is the diagonal matrix whose $(i,i)$ entry is 
   \begin{equation} \label{eq:iientry}
\left(-\frac{1}{4} \left(\frac{\mathrm{d}^2}{\mathrm{d} |z|^2}+ \frac{1}{|z|} \frac{\mathrm{d}}{\mathrm{d}|z|} \right) \chi u_{K,i,t} +
t^2|z|^{\frac{2}{K}} \left(\mathrm{e}^{\chi u_{K,i,t}-\chi u_{K,i+1,t}} - \mathrm{e}^{\chi u_{K,i-1,t}-\chi u_{K,i,t}} \right) \right) \de z \wedge \de \zbar.
 \end{equation}
(Note that without the cutoff function $\chi$, this vanishes.)
From the exponential decay of $u_{K,i}(|z|)|$ in $|z|$ 
like $\e^{-c |z|^{\frac{K+1}{K}}}$ (see Proposition \ref{prop:MK1model}), we 
see that---fixing $|z|$--- $u_{K, i, t}(|z|)$ decays in $t$ like $\e^{-ct}$. 
To see that the expression in \eqref{eq:iientry} is exponentially decaying, we break it into pieces. For $t\gg0$, there is a constant $C_1$ close to $1$ such that 
\begin{eqnarray}
 \left|\mathrm{e}^{\chi u_{K,i,t}-\chi u_{K,i+1,t}} - 1 \right| &\leq& C_1 \left(\chi u_{K,i,t}-\chi u_{K,i+1,t} \right) \leq C_1 \left( |u_{K,i,t}| + |u_{K,i+1,t}| \right)  \\ \nonumber 
  \left|\mathrm{e}^{\chi u_{K,i-1,t}-\chi u_{K,i,t}} - 1 \right| &\leq& C_1 \left(\chi u_{K,i-1,t}-\chi u_{K,i,t} \right) \leq C_1 \left( |u_{K,i-1,t}| + |u_{K,i,t}| \right).
  \end{eqnarray}
  Additionally, because $u_{K,i,t}$ and its derivatives in $|z|$ are all exponentially decaying in $t$, there is a constant $C_2$ depending on the maximum of  $|\chi'|$ and $|\chi''|$
  such that
  \begin{eqnarray}
\left|\left(\frac{\mathrm{d}^2}{\mathrm{d} |z|^2}+ \frac{1}{|z|} \frac{\mathrm{d}}{\mathrm{d}|z|} \right) \chi u_{K,i,t}  \right| &\leq& C_2 \e^{-ct}.
\end{eqnarray}
The exponential decay of \eqref{eq:iientry} follows.
\end{proof}

\section{Properties of the linearization}\label{sec:linearization}
In Proposition \ref{defpropreg}, we
proved that the family  $h_t^\app$ of approximate solutions was close to solving Hitchin's equations.
The metrics $h_t^\app$ failed to solve Hitchin's equations
only on the union of the gluing annuli around $p \in Z$ where the value of the cutoff function $\chi(|z_i|)$
differed from $0$ or $1$---and on those gluing annuli, the error was exponentially decaying in $t$.

Looking forward, the main  theorem (Theorem \ref{thm:main}) states something much stronger: for $t \gg 0$, 
the approximate metric $h_t^\app$
is close to the actual harmonic $h_t$ solving Hitchin's equations in the sense that 
\begin{equation} \label{eq:formulation}
h_t(v,w)=h_t^\app(\e^{-\gamma_t}\; v, \e^{-\gamma_t}\; w),
\end{equation}
for $\gamma_t$ ``small.''

 \medskip
Following the conventions of Mazzeo-Swoboda-Weiss-Witt, 
we do the analysis using the unitary formulation of Hitchin's equations discussed in \S\ref{sec:fixeddata}. We fix a hermitian metric $h_0$ on $E$.  We replace the triple $(\delbar_E, t\varphi, h_t^\app)$ with a pair $(\de_{A_t}, t \Phi_t)$ such that $(\delbar_E, t \varphi, h_t^\app)$ and $(\delbar_{A_t}, t \Phi_t, h_0)$ are complex gauge equivalent with respect to the action in \eqref{eq:complexgaugeequiv}, i.e. $[(\delbar_E, t \varphi, h_t^\app)]$ and $[(\delbar_{A_t}, t \Phi_t, h_0)]$ define the same point in the Hitchin moduli space $\cM$.

There are two other interesting actions of the complex gauge group on the space of triples $(\delbar_E, \varphi, h)$. For these, the equation $\delbar_E \varphi=0$ is preserved by the action of the complex gauge group; however, the equation $F^\perp_{D(\delbar_E, h)} + [\varphi, \varphi^{\dagger_h}]=0$ is not preserved.
In the first action,
the complex gauge group acts transitively on the space of hermitian metrics by 
\begin{equation}
 g \cdot_{1} (\delbar_E, \varphi, h) = (\delbar_E, \varphi, g \cdot h) \qquad \mbox{where $(g\cdot h)(v,w) = h(g v, g w)$}.
\end{equation}
If $(\delbar_E, \varphi)$ is polystable, then in the complex gauge orbit, there is a hermitian metric $g \cdot h$ solving Hitchin's equations.
In the second action, we fix the hermitian metric and take the action
\begin{equation}
 g \cdot_2 (\delbar_E, \varphi, h) = (g \circ  \delbar_E \circ g^{-1}, g \varphi g^{-1},  h).
\end{equation}
This second action induces  a complex gauge action on the space of pairs $(\de_A, \Phi)$
in the unitary formulation of Hitchin's equations:
\begin{equation} \label{eq:action2un}
g \cdot (\de_A, \Phi) = (D(g \circ \delbar_E \circ g^{-1}, h_0), g\Phi g^{-1} ),
\end{equation}
 where $D$ is the Chern connection associated to the pair.
 Note that these two actions 
 of the complex gauge transformation satisfy
 \begin{equation} \label{eq:actionrelation}
  g \cdot \left( g \cdot_2 (\delbar_E, \varphi, h) \right) = g \cdot_1 (\delbar_E, \varphi, h),
 \end{equation}
where $g \cdot$ is the action of the complex gauge transformation in \eqref{eq:complexgaugeequiv}.

\medskip

We are interested in finding the complex gauge transformation $g$ such that $g \cdot(\de_{A_t}, t\Phi_t)$ (defined in \eqref{eq:action2un})
solves Hitchin's equations.
Since Hitchin's equations are invariant under $h_0$-unitary gauge transformations,
we take the standard
slice of the complex gauge transformations modulo $h_0$-unitary gauge transformations by assuming that $g=\e^{-\gamma}$ is $h_0$-hermitian.
Define the operator 
\begin{equation}\label{eq:Fapp}
\mathbf{F}^{\app}_t(\gamma)
:=-i \star \left(F_{A_t^{\exp(-\gamma)}} + t^2[\e^{-\gamma}\Phi_t \e^{\gamma},\e^{\gamma} \Phi_t^{\dagger_{h_0}} \e^{-\gamma}] \right).
\end{equation}
(We add the superscript to clarify that this expression
is based at the approximate solution $(\delbar_E, t\varphi, h_t^\app)$.)
We are interested in the family $\gamma_t$ satisfying $\mathbf{F}^{\app}_t(\gamma_t)=0$.
Note that this is equivalent 
to finding an $h_0$-unitary $\gamma_t$ satisfying \eqref{eq:formulation}.

\bigskip

In this section, we study the linearization of $\mathbf{F}^\app_t$
and prove bounds on its inverse (Proposition \ref{prop:boundoninverse}).
The linearization of $\mathbf{F}^\app_t$ at $0$ is 
\begin{eqnarray} \label{eq:Lt}
 L_t\gamma:=D\mathbf{F}^\app_t(0)[\gamma]&=&\frac{d}{d\eps}\Big|_{\eps=0} \mathbf{F}^\app_t(\eps \gamma)\\ \nonumber
 &=& \Delta_{A_t} \gamma - i \star t^2 M_{\Phi_t} \gamma 
\end{eqnarray}
where 
\begin{eqnarray}
\Delta_{A_t} &:=& \de_{A_t}^* \de_{A_t} \gamma \\ \nonumber
 M_{\Phi_t}\gamma &:=& \left[\Phi_t^* \wedge [\Phi_t,  \gamma ] \right]
-\left[\Phi_t  \wedge [\Phi_t^*, \gamma] \right].
\end{eqnarray}

First note that $L_t$ is a positive operator.
\begin{prop}\label{prop:positivity}\cite[Proposition 5.1]{MSWW14}
 If $\gamma \in \Omega^0(\mathfrak{sl}(E))$, then
 \begin{equation}
  \IP{L_t \gamma, \gamma}_{L^2} = \norm{\de_A \gamma}_{L^2}^2 + 2t^2\norm{[\Phi, \gamma]}^2_{L^2} + 2t^2\norm{[\Phi^*, \gamma]}^2_{L^2} \geq 0.
 \end{equation}
 Consequently, restricted to $\Omega^0(\I \mathfrak{su}(E))$, $L_t$ has no kernel.
\end{prop}

We now prove that its inverse $L_t^{-1}: L^2(i \mathfrak{su}(E)) \rightarrow H^2(i \mathfrak{su}(E))$ is bounded.
\begin{prop}\label{prop:boundoninverse} 
For $t_0$ sufficiently large, there
is are constants $\tilde{C}_1,\tilde{C}_2>0$ such that
\begin{itemize}
\item[(a)] \hspace{1.5in} $ \|L_t^{-1}\|_{\cL(L^2, L^2)} \leq \tilde{C_1},$
\item[(b)] \hspace{1.5in} $\|L_t^{-1}\|_{\cL(L^2, H^2)} \leq \tilde{C_2}t^2.$
\end{itemize}
\end{prop}

\begin{rem}\label{rem:domaindecomp}
For the $SU(2)$ case, the analog of 
Proposition \ref{prop:boundoninverse}a is 
stated in \cite[Lemma 6.3]{MSWW14}.
An important ingredient of their strategy is the
the domain decomposition principle in \cite{Bar}.
They decompose $C$ into disjoint pieces: neighborhoods $\bD_p$ around
each point $p \in Z$, plus the remaining piece $C^{\rext}=C- \bigcup_p \bD_p$.
On each piece, 
they find a lower bound for the first Neumann eigenvalue.
Then,
the domain decomposition principle gives a
lower bound on the first global eigenvalue.

One might hope that this method of proof
works for $SU(n)$ when $n>2$.  However, this does
not work because the Neumann boundary problem
on each disk $\bD$ has kernel. 
By explicit computation of $L_t$ 
in the basis of \eqref{eq:regHapp}
on $\bD$ (see \ref{eq:decoupledtilde}),
one can compute
that the Neumann kernel consists
of constant traceless diagonal 
matrices
with the shape
\begin{equation}
 \gamma = \bigoplus_{j=1}^{m_p}  \gamma_{(j)}\mathbf{1}_{K_j} 
\end{equation}
Consequently, we pursue a global strategy that does not use the domain decomposition
principle. 
\end{rem}

\begin{proof}[Proof of Proposition \ref{prop:boundoninverse}a] 
\emph{(Necessary lemmata appear in \S\ref{app:localLt}.)}
Define
\begin{equation}
 \mathbb{L}_t\gamma := \Delta_{A_t} \gamma - i \star M_{\Phi_t} \gamma.
\end{equation}
Since $M_{\Phi_t}$ is a semi-positive operator, $L_t \geq \mathbb{L}_t$.  
Consider the eigenvalues $\{\lambda_\ell^t\}$ of $\mathbb{L}_t$. These are all positive by Proposition \ref{prop:positivity}.
 We will prove that the lowest eigenvalue $\lambda_0^t$ of $\mathbb{L}_t$ is bounded  below by some constant $\kappa>0$  as $t \rightarrow \infty$. Suppose to the contrary that $\lambda_0^t \rightarrow 0$.

We define a family of weight functions $\mu_t: C \rightarrow \R^{+}$ as follows:
Around each point $p \in Z$, work in the gauge from Proposition \ref{prop:reghol} \& \ref{prop:reggoodgaugehinfty} and let $z$ be some holomorphic coordinate centered at $p$.  (The coordinate $z$ need not be any of the holomorphic coordinates $z_i$ appearing in Proposition \ref{prop:reghol}.)
Order the elements of the partition of $n$ so 
that $K_1 \geq K_2 \geq \cdots \geq K_{m_p}$. 
\begin{figure}[h!]
 \begin{center}
  \includegraphics[height=1in]{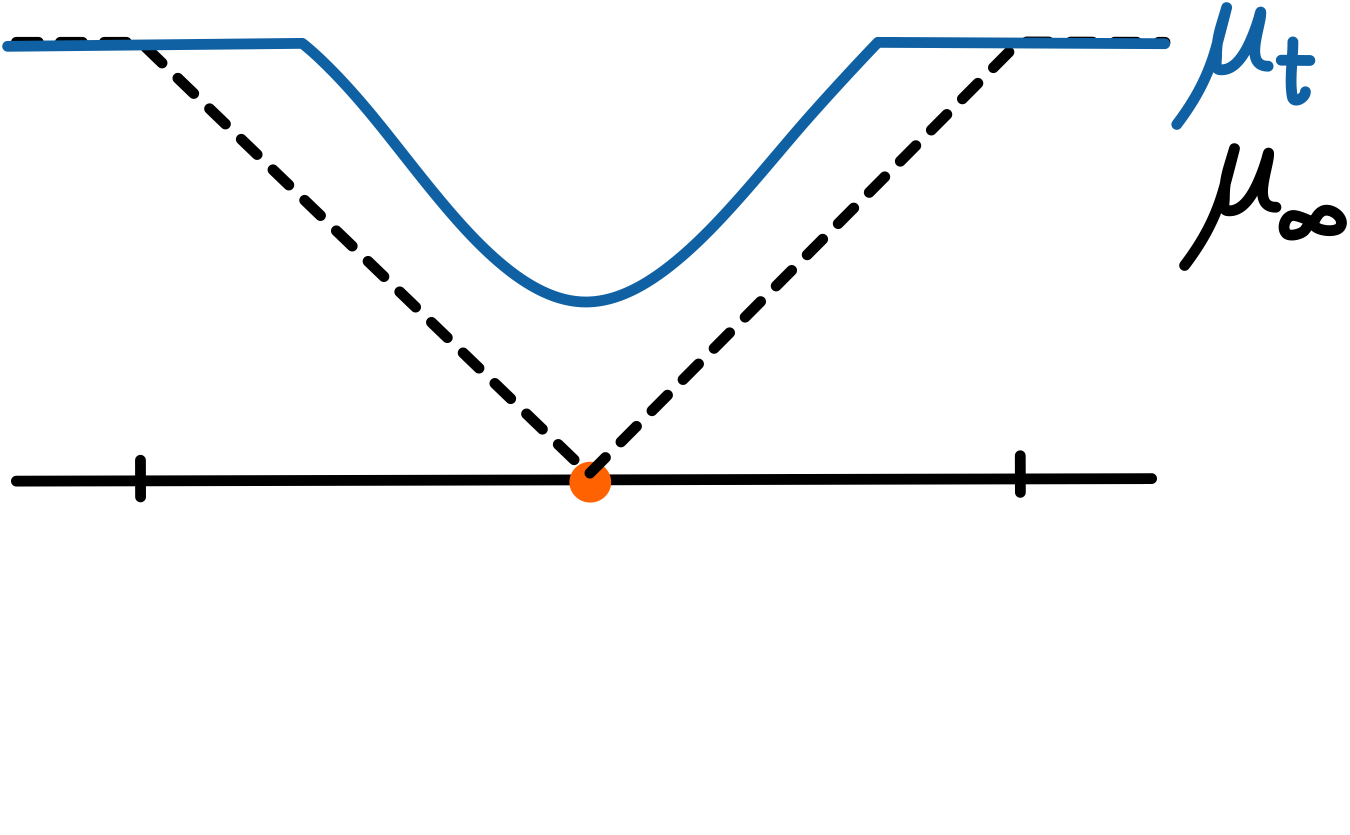}\\
  \vspace{-.5in}
  \caption{\label{fig:weight} Weight function $\mu_t$}
 \end{center}
\end{figure}
In the unit disk $\mathbb{D}_{p_j}=\{|z_j|\leq 1\}$ around $p_j$, define the weight function by
\begin{equation}
 \mu_t(z) :=  \min\left(  \left( t^{-\frac{2K_1}{K_1+1}} + |z|^2 \right)^{1/2}, 1  \right).
 \end{equation}
 On the rest of the surface, define 
 \begin{equation}
 \mu_t(x) :=  1 \qquad x \in C\setminus \{\mathbb{D}_{p_j}\}_{p_j \in Z}\end{equation}
  The weight function $\mu_t$ (shown in Figure \ref{fig:weight}) is continuous. (Its lack of regularity is immaterial, and we could easily introduce a smoothed version.)
  Note that $\mu_t$ increases in $|z|$ with minimum $\mu_t(0)= t^{-\frac{K_1}{K_1+1}}$.
 The family $\{\mu_t\}$ is uniformly bounded above by $1$, and the family is also bounded away from $0$ on any set where $|z|>\eps>0$.

 Let $\psi_t$ denote an eigensection of the first eigenvalue $\lambda_0^t$ . Fix some constant $\delta>0$. (We will choose a good value of $\delta$ later in the proof.)  We normalize $\psi_t$---multiplying it by a constant--- so that
\begin{equation} \label{eq:regnormalization}
 \sup_C \mu_t^\delta |\psi_t| =1.
\end{equation}
In what follows, we show the supremum of $\mu_t^\delta |\psi_t|$ cannot be achieved at any point of $C$---contradicting our initial assumption that $\lambda_0^t \rightarrow 0$.
 
\medskip
\noindent\textsc{Claim:} There is no value of $\eps \in (0,1]$ for which
there exists a constant $\eta$ such that for all $t$---or rather for some unbounded subsequence $\{t_k\}$---
\begin{equation} \label{eq:regclaim1}
 \sup_{C-\bigcup\limits_{p_j \in Z} \{|z_j| \leq \eps\}} \mu_t^\delta |\psi_t| \geq \eta >0.
\end{equation}\\
\emph{Proof of Claim:} $\triangleright$ 
Since \eqref{eq:regnormalization} holds, then for any choice of $\eps>0$
$|\psi_t| \leq  \mu_t^{-\delta}  \leq \eps^{-\delta}$ on ${C-\bigcup\limits_{p_j \in Z} \{|z_j| \leq \eps\}}$.
Because the eigensections $\{\psi_t\}$
are uniformly bounded in $L^\infty$
on $C-\bigcup_{p_j \in Z}\{|z_j|<\eps\}$, by compactness,
we may obtain a subsequence of $\psi_t$ which converges
in $L^\infty$ on the punctured surface $C-Z$ to a section
$\psi_\infty$; moreover,
by elliptic regularity, the $\{\psi_t\}$ and limiting 
$\psi_\infty$ are in $C^\infty$. 
In the region $C-\cup_Z \{|z_j| < \eps\}$,
the coefficients of $\mathbb{L}_t$ are converging smoothly;
thus $\psi_\infty$ satisfies
\begin{equation}
 \mathbb{L}_\infty \psi_\infty =0 \qquad \mbox{on $C- Z$}.
\end{equation}
Furthermore, $|\psi_\infty|$ is non-zero from our assumption in \eqref{eq:regclaim1}.

Using the local conic regularity theory at $p \in Z$ (see \cite{MazzeoEdge1, MazzeoWeiss}), $\psi_\infty$
has an asymptotic expansion in powers of $r=|z|$ with coefficients which are trigonometric functions of the
angular variable $\theta$.  
Because $\psi_\infty$ is bounded at $p \in Z$, all of the powers of $r$ in the expansion of $\psi_\infty$ at $p$ are nonnegative. 
Now, 
\begin{equation}
 \IP{\mathbb{L}_\infty \psi_\infty, \psi_\infty} =- \star  \de \star \IP{\de_{A_\infty} \psi_\infty, \psi_\infty} + \norm{\de_{A_\infty} \psi_\infty}^2  +2 \norm{[\Phi_\infty, \psi_\infty}^2.
\end{equation}
Note that $\star \IP{\de_{A_\infty} \psi_\infty, \psi_\infty}$ vanishes at each point $p \in Z$ (These reasons are elaborated in a more general setting in \eqref{eq:samereason}.);
hence, doing integration by parts, 
\begin{equation}
 0=\IP{L_\infty \psi_\infty, \psi_\infty}_{L^2(C)} = \norm{\de_{A_\infty} \psi_\infty}^2_{L^2(C)} + 2 \norm{[\Phi_\infty, \psi_\infty]}^2_{L^2(C)}.
\end{equation}
From Proposition \ref{prop:positivity}, there is no global non-zero solution satisfying both $\de_{A_\infty} \psi_\infty =0$ and $[\Phi_\infty, \psi_\infty]=0$, hence $\psi_\infty=0$.

Now, suppose the claim is false, i.e. suppose there is a choice of  $\eps \in (0,1]$ and $\eta>0$ such that for all $t$
\begin{equation} \label{eq:supposenot}
 \sup_{C-\bigcup\limits_{p_j \in Z} \{|z_j| \leq \eps\}} \mu_t^\delta |\psi_t| \geq \eta.
\end{equation}
Then, $\sup_{C-\bigcup\limits_{p_j \in Z}\{|z_j| \leq \eps\}} |\psi_t| \geq  \sup_{C-\bigcup\limits_{p_j \in Z}\{|z_j| \leq \eps\}} \eta \mu_t^{-\delta} \geq \eta,$ so $\psi_\infty$ is non-zero, a contradiction. 
 $\triangleleft$

\bigskip

Let  $\{q_t\}\rightarrow \overline{q}$ be a convergent sequence of  points at which the supremum of $\mu_t^\delta |\psi_t|$ in \eqref{eq:regnormalization} is achieved. From the \textsc{Claim}, we see that $\overline{q}=p \in Z$ and that $\mu_t^\delta |\psi_t|$ tends to zero pointwise (and in fact uniformly in compact subsets) on $C-Z$ . 
Let $z:\bD \rightarrow \C$ be the chosen holomorphic coordinate centered at $p$.
Define $z_t:=z(q_t)$. Note that $\{z_t\}$ converge to zero because $\{q_t\}$ converge
to $p$. Let $K_0$ be the largest integer less than or equal to $K_1$ for which
$ t^{\frac{K_0}{K_0+1}} z_t$
is bounded above by some constant $R$.
Note that $K_0$ is automatically nonnegative.
We will now show that if we assume that $\lambda_0^t \rightarrow 0$, then the supremum of \eqref{eq:regnormalization} also can't be achieved at a point $p \in Z$, by separately considering two cases, depending on whether
$z_t$ converge to zero more quickly (\textsc{Case A}, $K_0>0$) or more slowly (\textsc{Case B}, $K_0=0$).

\medskip

\noindent{\textsc{Case A. $K_0>0$}}: 
Define a rescaling
\begin{equation}
 \rho_{K_0,t}:z \rightarrow t^{\frac{K_0}{K_0+1}} z=:w.
\end{equation}
Let $w_t=\rho_{K_0,t}(z_t)$, and note that $|w_t| \leq R$.
Now, pullback and rescale the eigensections $\psi_t$, taking 
\begin{equation}\Psi_t := t^{-\frac{\delta K_0}{K_0+1}}(\rho_{K_0,t}^{-1})^* \psi_t.\end{equation}
On the disk, the bound in \eqref{eq:regnormalization} is
\begin{equation}
 \left( t^{-\frac{2K_1}{K_1+1}} + |z|^2 \right)^{\delta/2} |\psi_t(z)| \leq 1.
\end{equation}
This implies that 
\begin{equation}
 \left(t^{-\frac{2K_1}{K_1+1}+\frac{2K_0}{K_0+1}}+ |w|^2\right)^{\delta/2} |\Psi_t(w)| \leq 1, 
\end{equation}
with equality attained at $w_t$. 
Since $K_0 \leq K_1$, $-\frac{2K_1}{K_1+1}+\frac{2K_0}{K_0+1} \leq 0$.
Consequently,
\begin{equation}\label{eq:regbound} 
|\Psi_\infty(w)| \leq
\begin{cases}\left(1+ |w|^2 \right)^{-\delta/2} \qquad &\mbox{if }K_0 = K_1\\
                        |w|^{-\delta} \qquad &\mbox{if }K_0  < K_1.\\
                       \end{cases}
\end{equation}
Since the disk $\{|w| \leq R\}$ is compact,
a subsequence of $w_t$ converges to some $\overline{w}$, hence $|\Psi_\infty(\overline{w})| \neq 0$.
By Lemma \ref{lem:3laplacians}, 
\begin{equation}\label{eq:operatorlimit}
 \hspace{-.2in} \lim_{t \rightarrow \infty} t^{-\frac{2K_0}{K_0+1}} (\rho_{K_0,t}^{-1})^*\mathbb{L}_t =\Delta_{\widetilde{A}}= \left(\bigoplus\limits_{j: \; K_j > K_0} A_\infty \right) \oplus \left(\bigoplus\limits_{j: \; K_j = K_0} A_{\mathrm{mod}}\right) \oplus \left(\bigoplus\limits_{j:\; K_j < K_0} A_0\right).
\end{equation}
The expressions for $A_\infty, A_{\mathrm{mod}}$, and $A_0$ are given in \eqref{eq:lapshape}.
Because the coefficients of the operators in \eqref{eq:operatorlimit} converge smoothly to $\Delta_{\widetilde{A}}$, the non-zero $\Psi_\infty$ 
satisfies the bound in \eqref{eq:regbound} and 
\begin{equation}\label{eq:reg2a}
 \Delta_{\widetilde{A}} \Psi_\infty =0. 
\end{equation}
By Proposition \ref{prop:nosoln}, for $\delta>0$ sufficiently small, there is no non-zero solution.
But $\Psi_\infty$ is non-zero! Thus, \textsc{Case A} cannot hold.

\medskip

\noindent\textsc{Case B. $K_0=0$:} Lastly,
 suppose that $|\rho_{1,t}(z_t)|=|t^{1/2}z_t|$ is unbounded.
Define a rescaling 
\begin{equation}
 \sigma_t:z \rightarrow  |z_t|^{-1} z=:\widetilde{w}
\end{equation}
Let $\widetilde{w}_t = \sigma_t(z_t)$, and note that $|\widetilde{w}_t|=1$.
Now, pullback the eigensection $\psi_t$ and rescale it by an aptly chosen constant
\begin{equation}
 \widetilde{\Psi}_t:= |z_t|^\delta  (\sigma_t^{-1})^* \psi_t 
\end{equation}
The constant  is chosen so that the bound in \eqref{eq:regnormalization} implies that
\begin{equation}\label{eq:regbound2b}
 (|z_t|^{-2}t^{-\frac{2 K_1}{K_1+1}}+|\widetilde{w}|^2)^{\delta/2} |\widetilde{\Psi}_t(\widetilde{w})| \leq 1.
\end{equation}
 Taking the limit of the bounds in \eqref{eq:regbound2b}, we see that 
 $\widetilde{\Psi}_\infty$ satisfies
\begin{equation}\label{eq:regbound3}
 |\widetilde{\Psi}_\infty(\widetilde{w})| \leq |\widetilde{w}|^{-\delta}.
\end{equation}
Since $\widetilde{\Psi}_t$ achieves the bound in \eqref{eq:regbound2b} at $\widetilde{w}_t$ which has norm one,
$\widetilde{\Psi}_\infty$ also achieves the bound in \eqref{eq:regbound3} on the unit circle; hence $\widetilde{\Psi}_\infty$ is non-zero.

In the rescaling limit,
\begin{equation}
  \lim_{t \rightarrow \infty} (\sigma_t^{-1})^* M_{\Phi_t} = M_{\Phi_{\infty}} \qquad 
  \lim_{t \rightarrow \infty} |z_t|^{-2} (\sigma_t^{-1})^*\Delta_{A_t} =\Delta_{A_\infty};
\end{equation}
consequently, 
\begin{equation}
 \lim_{t \rightarrow \infty} |z_t|^{-2} (\sigma_t^{-1})^*\mathbb{L}_t =\Delta_{A_\infty},
\end{equation}
where $A_\infty$ is defined in \eqref{eq:lapshape}.
Thus, $\widetilde{\Psi}_\infty$ is non-zero and satisfies
\begin{equation}
 \Delta_{A_\infty} \widetilde{\Psi}_\infty(\widetilde{w})=0, \qquad |\widetilde{\Psi}_\infty(\widetilde{w})| \leq |\widetilde{w}|^{-\delta}.
\end{equation}
By Proposition \ref{prop:nosoln}, for $\delta>0$ sufficiently small, there is no non-zero solution.
But $\widetilde{\Psi}_\infty$ is non-zero. 
Thus, \textsc{Case B} too is impossible.

\medskip

In summary, we have shown that it is impossible that $ \lambda_0^t \rightarrow 0$.
\end{proof}

\begin{proof}[Proof of Proposition \ref{prop:boundoninverse}b]

The proof that $\|L_t^{-1}\|_{\cL(L^2, H^2)} \leq \tilde{C} t^2$
is a direct adaption of the proof of in the $SU(2)$ case
\cite[Lemma 6.5]{MSWW14}.
The
graph norm of $\Delta_{A_\infty}$ is equivalent to the
standard Sobolev $H^2$-norm \cite[Lemma 6.5]{MSWW14}.
Consequently, we will prove that there is a constant $\tilde{C}'$ such that
\begin{equation}\label{eq:goal1}
 \sqrt{\|L_t^{-1} u\|^2_{L^2} + \|\Delta_{A_\infty} L_t^{-1} u\|^2_{L^2}}  \leq \tilde{C}'t^2\|u\|_{L^2}.
\end{equation}
Define
\begin{equation}
 \widetilde{L}_t = \Delta_{A_\infty} - \I \star t^2 M_{\Phi_\infty}.
\end{equation}
(For comparison,
recall from \eqref{eq:Lt} that
$ L_t=\Delta_{A_t} - \I \star t^2 M_{\Phi_t}. $)
Then, note that
\begin{eqnarray}
 \|\Delta_{A_\infty} L^{-1}_t u\|_{L^2} &\leq& \|u\|_{L^2} +  
 \|(\Delta_{A_\infty}-L_t) L_t^{-1}  u\|_{L^2} \\ \nonumber
 &\leq&  \|u\|_{L^2} +   \left(t^2
 \|M_{\Phi_\infty}\|_{\cL(L^2,L^2)}+\|(\widetilde{L_t}-L_t)\|_{\cL(L^2,L^2)} \right)\|L_t^{-1}\|_{\cL(L^2,L^2)} \|u\|_{L^2}.
\end{eqnarray}
The bound
 $\norm{M_{\Phi_\infty}}_{\cL(L^2, L^2)} \leq c_M$ follows in Lemma \ref{lem:collected}.
 The bound
$\norm{\widetilde{L_t}-L_t}_{\cL(L^2, L^2)} \leq C \e^{-\delta t}$ follows because $\Phi_t$ converges to $\Phi_\infty$ exponentially in $t$ and $A_t$ converges to $A_\infty$ exponentially in $t$. (See \cite[Lemma 6.5]{MSWW14} for the case of $K\times K =2 \times 2$ blocks.)
 The bound $\|L_t^{-1}\|_{\cL(L^2,L^2)} \leq \widetilde{C}_1$ is from Proposition \ref{prop:boundoninverse}a. Thus we obtain the
desired bound in \eqref{eq:goal1}.
\end{proof}

In the proof of Proposition \ref{prop:boundoninverse}b, we used the following bound
on $M_{\Phi_\infty}$. 
\begin{lem}\label{lem:collected}
 There is a constant $c_M$ such that at any point of $C$
 \begin{equation} \label{eq:Mbound}
  |M_{\Phi_\infty}|_{g_C, h_\ell} \leq c_M.
 \end{equation}
\end{lem}
\begin{proof}
Over $\Sigma$, $\pi^* \cE$ decomposes as the sum
of eigenline bundles of $\pi^*\varphi$. 
Let $\cL_i$
be the line bundle corresponding to globally-defined eigenvalue $\pi^*\lambda_i$.
 To see the bound on $M_{\Phi_\infty}$, note
 that, pulled-back from $C$ to $\Sigma$, $\pi^*\End\; \cE=\oplus \mathrm{Hom}(\cL_i, \cL_j)$ and the $(i,j)$-entry of $\pi^*M_{\Phi_\infty}$ is
 \begin{equation} 
 \left( \pi^*M_{\Phi_\infty}\gamma \right)_{ij}= 2|\lambda_i-\lambda_j|^2 \gamma_{ij}.
 \end{equation}
The difference between the eigenvalues of $\varphi$ are bounded above,
hence $|M_{\Phi_\infty}|$ is bounded.
\end{proof}

\subsection{Lemmata for Proposition \ref{prop:boundoninverse}a: Local analysis \texorpdfstring{$\mbox{of } \Delta_{\widetilde{A}}$}{ }}
\label{app:localLt}

Take $p \in Z$ with associated partition $n=K_1+ \cdots K_{m_p}$ and holomorphic coordinates $z_1, \cdots, z_{m_p}$ centered at $p$ in Proposition \ref{prop:reghol}. Given a holomorphic coordinate
$z$ centered at $p$, define biholomorphic functions $f_i = z_i \circ z^{-1}$ such that $f_i(0)=0$.  Given a choice of positive number $J$, define 
\begin{equation} \label{eq:Atilde}
\widetilde{A}:= \left(\bigoplus\limits_{j: \; K_j > J} A_\infty \right) \oplus \left(\bigoplus\limits_{j: \; K_j = J} A_{\mathrm{mod}} \right)  \oplus \left(\bigoplus\limits_{j: \; K_j < J} A_{0} \right),
 \end{equation}
 where the $K\times K$ blocks are 
\begin{eqnarray}\label{eq:lapshape}
A_\infty&=& \begin{pmatrix}
   - \frac{\alpha_{K,1}}{2}  & & \\
      & \ddots & \\
      & &  -\frac{\alpha_{K,K}}{2} \end{pmatrix}_{K \times K} \hspace{-.4in}\left( \frac{\de z}{z} - \frac{\de \zbar}{ \zbar} \right)\\ \nonumber
A_{\mathrm{mod}}
   &=&\begin{pmatrix}
   - \frac{\alpha_{K,1}}{2} + \frac{|z|}{4} \frac{\de u_{K,1,1}\left(|f'(0)| |z|\right)}{\de |z|} & & \\
      & \ddots & \\
      & &  -\frac{\alpha_{K,K}}{2} +\frac{|z|}{4} \frac{\de u_{K,K,1}\left(|f'(0)| |z|\right)}{\de |z|} \end{pmatrix}_{K \times K} \hspace{-.4in}\left( \frac{\de z}{z} - \frac{\de \zbar}{\zbar} \right)\\ \nonumber
A_0&=& \mathbf{0}_{K \times K}.
\end{eqnarray}
   \begin{figure}[ht]
  \begin{center}
   \includegraphics[height=.75in]{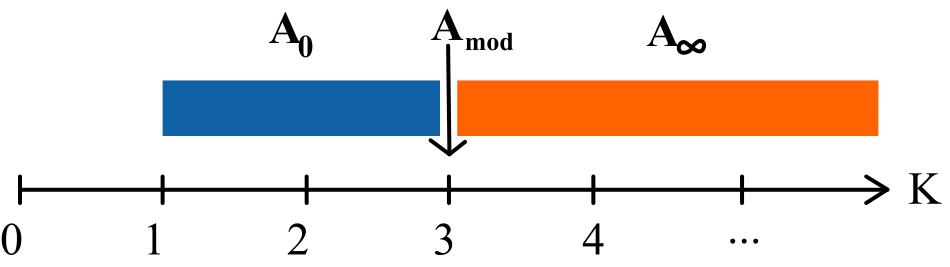} 
   \caption{\label{fig:laplacian} The values $K_1, \cdots, K_{m_p}$ are separated into three different categories: less than, equal to, or greater than some critical integer (here, 3). 
   The limiting $\widetilde{A}$  consequently features three types of blocks: $A_0, A_{\mathrm{mod}}, A_\infty$.}
  \end{center}
 \end{figure}
 
\begin{ex}
 For example, if $n=2+1+1$,
 \begin{equation}
\widetilde{A} = \left(\frac{1}{8} +\frac{|z|}{4} \frac{du_1}{d|z|} \right) \begin{pmatrix}1 & 0 &  0 &  0 \\ 0 & -1 & 0 & 0 \\  0 & 0 & 0 & 0 \\ 0 & 0 & 0  &0 \end{pmatrix} 
\left( \frac{\de z}{z} - \frac{\de \zbar}{\zbar} \right), 
 \end{equation}
where $u_1=u_{2,1,t=1}$ is the function in \eqref{eq:Painleve}.
\end{ex}

In this section, we prove that $\widetilde{A}$ appears naturally in a rescaling limit (Lemma \ref{lem:3laplacians}). This is used in the proof of Proposition \ref{prop:boundoninverse}a in \eqref{eq:operator}.
We then prove in Proposition \ref{prop:nosoln}  that for $\delta>0$ sufficiently small, there are no solutions of
\begin{equation}
 \Delta_{\widetilde{A}} \Psi=0 \qquad |\Psi| \leq |z|^{-\delta}.
\end{equation}

 \begin{lem}\label{lem:3laplacians}
 Define a rescaling
 \begin{equation}
 \rho_{J,t}: z \mapsto t^{\frac{J}{J+1}}z.
\end{equation}
Then, for the approximate solution $(A_t, \Phi_t)$
on $\bD$ in the unitary gauge of Remark \ref{rem:regunitarygauge}, we have
 \begin{equation}\label{eq:operatorlimitlem}
 \lim_{t \rightarrow \infty} t^{-\frac{2J}{J+1}} (\rho_{J,t}^{-1})^* \Delta_{A_t} =\Delta_{\widetilde{A}}.
\end{equation}
\end{lem}

\begin{proof}
First note that for any holomorphic function $f$ such that $f(0)=0$, we have
\begin{equation}
 \lim_{t \to \infty} (f \circ \rho_{J,t}^{-1})^* \left( \frac{\de z}{z} - \frac{\de \zbar}{\zbar} \right) = \left( \frac{\de z}{z} - \frac{\de \zbar}{\zbar} \right).
\end{equation}
This follows from expanding the holomorphic---hence analytic---function $f$ as $f(x) = \sum_{i=0}^\infty \frac{1}{i!}f^{(i)}(0) x^i$
in the following expression:
\begin{equation}
\lim_{t \to \infty} \frac{\de(f (t^{-\frac{J}{J+1}}w))}{f (t^{-\frac{J}{J+1}}w)}=
\frac{\lim_{t \to \infty}f'(t^{-\frac{J}{J+1}}w)  \de w}{ t^{\frac{J}{J+1}}f(t^{-\frac{J}{J+1}}w)}= \frac{f'(0) \de w}{f'(0)w}= \frac{\de w}{w}.
\end{equation}
Secondly, note that for similar reasons
\begin{equation}
 \lim_{t \to \infty} (f \circ \rho_{J,t}^{-1})^*  \left(\frac{|z|}{4} \frac{\de}{\de |z|} \right) = \frac{|z|}{4} \frac{\de}{\de |z|}.
\end{equation}
(Here, it is convenient to use that $|z| \frac{\del}{\del |z|} = z \frac{\del}{\del z} + \zbar \frac{\del}{\del \zbar}$.)

We work separately in each $K \times K$ block.
The computation of the limits is based on the following observation:
\begin{equation}
(\rho_{J,t}^{-1})^* u_{K,i,t} (r) = u_{K,i,t} \left(t^{-\frac{J}{J+1}} r \right) = 
u_{K,i,t} \left(t^{-\frac{K}{K+1}} \; t^{\frac{K}{K+1}-\frac{J}{J+1}}r \right)=
u_{K,i,t=1} \left( t^{\frac{K}{K+1}-\frac{J}{J+1}}r \right)
\end{equation}

We see that if $J<K$, then $\frac{K}{K+1}-\frac{J}{J+1}>0$, so $\lim_{t \to \infty}  t^{\frac{K}{K+1}-\frac{J}{J+1}}r=\infty$. Each function $u_{K, i, t}$ decays to $0$ at $\infty$, hence
\begin{equation}
 \lim_{t \to \infty} (f \circ \rho_{J,t}^{-1})^* \left( \frac{|z|}{4} \frac{\de u_{K,i,t}(|z|)}{\de |z|} \right)= 0,
\end{equation}
and consequently, in the limit of \eqref{eq:operatorlimitlem},
the $K \times K$ block is $A_\infty$.

If $J=K$, then $\frac{K}{K+1}-\frac{J}{J+1}=0$, so $\lim_{t \to \infty}  t^{\frac{K}{K+1}-\frac{J}{J+1}}z=z$. 
Additionally, note that because $f$ is analytic 
\begin{eqnarray}
 \left(f \circ \rho_{K,t}^{-1} \right)^{*} u_{K,i,t}(|z|) &=&u_{K,i,t}\left(\left|f(t^{-\frac{K}{K+1}} z)\right|\right)\\ \nonumber &=& 
  u_t\left(\left|\sum_{k=0}^\infty \frac{1}{k!}\,f^{(k)}(0) \, (t^{-\frac{K}{K+1}} z)^k \right|\right)\\ \nonumber
  &=&  
  u_1\left(\left|\sum_{k=1}^\infty \frac{1}{k!}\,f^{(k)}(0)\, (t^{-\frac{K}{K+1}})^{k-1} \, z^k \right|\right)
\end{eqnarray}
In the last line, we use that $u_t(t^{-\frac{K}{K+1}} r) = u_1(r)$ and $f(0)=0$.
Consequently, taking the limit, we obtain
\begin{equation}
 \lim_{t \to \infty} ( u_t\left(\left|f(t^{-\frac{K}{K+1}} z)\right|\right) = \lim_{t \to \infty} 
 u_1\left(\left|\sum_{k=1}^\infty \frac{1}{k!}\,f^{(k)}(0)\, (t^{-\frac{K}{K+1}})^{k-1} \, z^k \right|\right) =  u_1\left(\left|f'(0)\right| \left| z \right|\right).
\end{equation}
Hence, 
\begin{equation}
 \lim_{t \to \infty} (f \circ \rho_{J,t}^{-1})^* \left( \frac{|z|}{4} \frac{\de u_{K,i,t}(|z|)}{\de |z|} \right)= \left( \frac{|z|}{4} \frac{\de u_{K,i,t=1}(|f'(0)|\;|z|)}{\de |z|} \right);
\end{equation}
in the limit of \eqref{eq:operatorlimitlem},
the $K \times K$ block is $A_\mathrm{mod}$.

Lastly, if $J>K$, then $\frac{K}{K+1}-\frac{J}{J+1}<0$, so $\lim_{t \to \infty}  t^{\frac{K}{K+1}-\frac{J}{J+1}}z=0$. Since $\frac{|z|}{4} \frac{\de u_{K,i,1}(0)}{\de |z|} = \frac{\alpha_{K,i}}{2}$, we have
 \begin{equation}
  \lim_{t \to \infty} (f \circ \rho_{J,t}^{-1})^* \left(- \frac{\alpha_{K,i}}{2} +  \frac{|z|}{4} \frac{\de u_{K,i,t}(|z|)}{\de |z|} \right)
  = 0;
 \end{equation}
in the limit of \eqref{eq:operatorlimitlem},
the $K \times K$ block is $A_0$.
\end{proof}

We now analyze the kernel of $\Delta_{\widetilde{A}}$.
The operator 
$\Delta_{\widetilde{A}}$ on the punctured plane $\C\setminus\{0\}$ is an differential edge operator\footnote{This is roughly because, in polar coordinates, the operator $r^2 \Delta_{\widetilde{A}}$ can be written in terms of $r \del_r$ and $\del_\theta$.}, and so the theory in \cite{MazzeoEdge1, MazzeoWeiss} applies.
We first compute the indicial roots of the operator.
   A number $\nu \in \C$ is called an indicial root for $\Delta_{\widetilde{A}}$
   if there exists some function $\zeta = \zeta(\theta)$ such that
   \begin{equation}
    \Delta_{\widetilde{A}}(r^\nu \zeta(\theta)) = O(r^{\nu-1}),
   \end{equation}
rather than the expected rate $O(r^{\nu-2})$ \cite[Definition 4.2]{MSWW14}.
If $\Delta_{\widetilde{A}} \psi=0$, we see that
$\psi$ has a inhomogeneous asymptotic development around $0$ (or $\infty$)
in terms of the indicial roots of $\Delta_{\widetilde{A}}$ at $0$ (or, respectively, $\infty$).

\begin{lem} \label{lem:indicialroots}
In this basis of $\mathfrak{sl}(n,\C)$,
 the operator $\Delta_{\widetilde{A}}$ fully decouples. On in $(i,j)$ block, $\Delta_{\widetilde{A}}$ acts as
 \begin{equation} \label{eq:decoupledtilde}
  \Delta_{\widetilde{A}} \gamma_{ij} = \left(\de + (\widetilde{A}_{ii} -\widetilde{A}_{jj}) \right)^* \left(\de + (\widetilde{A}_{ii} -\widetilde{A}_{jj}) \right) \gamma_{ij}.
 \end{equation}
The set indicial roots of 
\begin{equation}
 \Delta_{\widetilde{A}}: \Gamma(\I\mathfrak{su}(n)) \rightarrow  \Gamma(\I\mathfrak{su}(n))
\end{equation}
at $|z|=0$  is  $\Gamma(\Delta_{\widetilde{A}}, 0)=\Z \sqcup S_0$,
where $S_0$ is a discrete set, symmetric around the origin, with $S_0 \subset (-1, 1)$.
Similarly, the indicial roots at $|z|=\infty$ is $\Gamma(\Delta_{\widetilde{A}}, \infty)=\Z \sqcup S_\infty$,
where $S_\infty$ is a discrete set, symmetric around the origin, with $S_\infty \subset (-1, 1)$.
\end{lem}
\begin{proof}
The operator
$\Delta_{\widetilde{A}}$ decouples as in \eqref{eq:decoupledtilde} because $\widetilde{A}$ is diagonal. Because $\widetilde{A}_{jj} = 2\I f_i(r) \de \theta$, 
we can compute  that
\begin{eqnarray} \label{eq:decoupled1}
  \Delta_{\widetilde{A}} \gamma_{ij} &=& \left(\de + (\widetilde{A}_{ii} -\widetilde{A}_{jj}) \right)^* \left(\de + (\widetilde{A}_{ii} -\widetilde{A}_{jj}) \right) \gamma_{ij}\\ \nonumber
 &=&
 r^{-2} \left( (r \del_r)^2  + \left( \del_ \theta + 2\I (f_i-f_j) \right)^2 \right) \gamma. 
\end{eqnarray}

To compute the indicial roots at $0$ for $\Delta_{\widetilde{A}}$ acting on $\Gamma(\mathfrak{sl}(n,\C)$, we only need to look at the highest order part of $\Delta_{\widetilde{A}}$ at $0$. We evaluate the function $f_i(r)-f_j(r)$ appearing in \eqref{eq:decoupled1} at $r=0$.
In a $K \times K$ block of $A_\infty$, $f_i(0)=-\frac{\alpha_{K,i}}{2}$; for both $A_{\mathrm{mod}}$ and $A_0$, this constant is $0$.
Then, taking $b_{ij}=f_i(0)-f_j(0)$, we see the relevant operator is 
\begin{eqnarray} \label{eq:operator}
 \left( (r \del_r)^2+  \left( \del_ \theta + 2\I b_{ij} \right)^2 \right) \gamma.
\end{eqnarray}
 Suppose $\nu$ is an indicial root at $0$ in $(i,j)$ block. 
 Then there is some function 
 $\zeta(\theta) = \sum_{\ell \in \Z} a_\ell \e^{i \ell \theta}$ such that 
 $\left.\Delta_{\widetilde{A}}\right|_{ij} (r^\nu \zeta(\theta)) = O(r^{\nu-1})$. Taking $\zeta(\theta) = \e^{\I \ell \theta}$,
\begin{eqnarray}
 0
 &=& \left( (r \del_r)^2 + \left( \del_ \theta + 2\I b_{ij} \right)^2 \right)(r^\nu \e^{i \ell \theta})\\ \nonumber
  &=& r^\nu \e^{i \ell \theta} \left( \nu^2 - ( \ell + 2 b_{ij})^2 \right).
\end{eqnarray}
Consequently, since $\ell \in \Z$, $\nu \in \{\Z + 2b_{ij}\} \cup \{\Z - 2b_{ij}\}$.

Further restricting to $\I\mathfrak{su}(n)$, we note that $\gamma_{ji} = \overline{\gamma}_{ij}$. We compute the indicial roots at $0$ for the direct sum of the $(i,j)$-block with the $(j,i)$-block. Letting $\gamma_{ij} = r^\nu \zeta(\theta)$, and taking $\zeta(\theta) = a_\ell \e^{i \ell \theta} + a_{-\ell} \e^{-i \ell \theta}$ we have 
\begin{eqnarray}
 \left.\Delta_{\widetilde{A}}\right|_{ij \oplus ji} \begin{pmatrix} r^\nu \zeta(\theta) \\ r^\nu \overline{\zeta(\theta)}\end{pmatrix}  &=&
 \begin{pmatrix} 
  \left( r^2 \del_r^2  + r \del_r + \left( \del_ \theta + 2\I b_{ij} \right)^2 \right) r^\nu (a_\ell \e^{i \ell \theta} + a_{-\ell} \e^{-i \ell \theta}) \\ 
   \left( r^2 \del_r^2  + r \del_r + \left( \del_ \theta - 2\I b_{ij} \right)^2 \right) r^\nu (\overline{a}_\ell \e^{-i \ell \theta} + \overline{a}_{-\ell} \e^{i \ell \theta})
 \end{pmatrix}\\
 &=& \begin{pmatrix}
      r^\nu a_{\ell}\e^{i \ell \theta} \left( \nu^2 - ( \ell + 2 b_{ij})^2 \right)
      + r^\nu a_{-\ell}\e^{-i \ell \theta} \left( \nu^2 - (-\ell + 2 b_{ij})^2 \right)\\ \nonumber
            r^\nu \overline{a}_{\ell}\e^{-i \ell \theta} \left( \nu^2 - ( -\ell + 2 b_{ij})^2 \right)
      + r^\nu \overline{a}_{-\ell}\e^{i \ell \theta} \left( \nu^2 - (\ell + 2 b_{ij})^2 \right)
     \end{pmatrix}
\end{eqnarray}
Thus, we see that $\nu \in \{\ell + 2b_{ij}, -\ell - 2b_{ij}\} \cap  \{-\ell + 2b_{ij}, \ell - 2b_{ij}\}$. 
Thus, as claimed:
\begin{itemize}
 \item If $b_{ij}=0$, then $\nu \in \Z$; 
 the indicial root $\nu$ comes from
 $r^\nu \zeta(\theta)$ where $\zeta(\theta) = a_\nu \e^{\I  \nu \theta} + a_{-\nu} \e^{-\I \nu \theta}$. (Note that the $\nu=0$ indicial root comes from $a_0 + \widehat{a}_0 \log r$.)
 \item If $b_{ij} \neq 0$, then $\nu \in \{\pm 2b_{ij}\}$;
 the indicial root $\nu$ comes from
 $r^{\nu} \zeta(\theta)$
 where $\zeta(\theta)$ is constant. 
\end{itemize}
It is worth noting that because $\alpha_{K,i} \in (- \frac{1}{2}, \frac{1}{2})$, 
we automatically have that $b_{ij} \in (- \frac{1}{2}, \frac{1}{2})$ as well.
Hence, the indicial roots of $\Delta_{\widetilde{A}}$  at $0$ are  $\Z \sqcup S_0$,
where $S_0$ is a discrete set, symmetric around the origin, with $S_0 \subset (-1, 1)$.

\bigskip

The computation of the indicial roots at $\infty$ is similar.
We let $v=z^{-1}$ and introduce polar coordinates $v=s\e^{\I \vartheta}$. In these coordinates, the operator in \eqref{eq:decoupled1} is  
\begin{equation}
 r^2  \Delta_{\widetilde{A}} \gamma_{ij} =  
 \left( (s \del_s)^2  + \left( -\del_\vartheta + 2\I (f_i-f_j) \right)^2 \right) \gamma.
\end{equation}
(Here, we used that $r \del_r = -s \del_s$ and $\del_\theta = -\del_\vartheta$.)
The computation of the indicial roots is similar.
In a $K \times K$ block of $A_\infty$ or $A_{\mathrm{mod}}$, $f_i(\infty)=-\frac{\alpha_{K,i}}{2}$; for $A_0$, this constant is $0$.
Define the constant $c_{ij}=\left.f_i-f_j\right|_{r=\infty}$.
Thus, as claimed:
\begin{itemize}
 \item If $c_{ij}=0$, then $\nu \in \Z$; 
 the indicial root $\nu$ comes from
 $s^\nu \zeta(\vartheta)$ where $\zeta(\vartheta) = a_\nu \e^{\I  \nu \vartheta} + a_{-\nu} \e^{-\I \nu \vartheta}$.
 \item If $c_{ij} \neq 0$, then $\nu \in \{\pm 2c_{ij}\}$;
 the indicial root $\nu$ comes from
 $s^{\nu} \zeta(\vartheta)$ where $\zeta(\vartheta) = a_0.$
\end{itemize}
Hence, the indicial roots of $\Delta_{\widetilde{A}}$ at $r=\infty$ are $\Z \sqcup S_\infty$,
where $S_\infty$ is a discrete set, symmetric around the origin, with $S_\infty \subset (-1, 1)$.
\end{proof}

Using the computation of the indicial roots in Lemma \ref{lem:indicialroots}, we now prove:

\begin{prop}\label{prop:nosoln}
There exists a $\delta>0$ such that there are no non-zero solutions $\psi \in \Gamma(i \mathfrak{su}(n))$ solving
\begin{equation} \label{eq:nosoln}
 \Delta_{\widetilde{A}} \psi =0 \qquad |\psi | \leq (\eps +  |z|^2)^{-\delta/2}
\end{equation}
where $\eps =0,1$.
\end{prop}

\begin{proof}
In the proof, we will integrate by parts and then conclude that there is no solution of $\de_{\widetilde{A}} \psi =0$ solving the above bound. Because
\begin{equation}\label{eq:IBP3}
 \star \IP{\Delta_{\widetilde{A}} \psi, \psi} = \star \norm{\de_{\widetilde{A}} \psi}^2 +  \de \star \IP{ \de_{\widetilde{A}} \psi, \psi},
\end{equation}
we will additionally need to show that 
$\lim_{r \to 0, \infty}
\int_{S^1_r} \star \IP{\de_{\widetilde{A}} \psi, \psi}=0$.

Choose $\delta$ satisfying
\begin{equation}
 \delta < \min\{ |\nu| : \nu \in S_0 \cup S_\infty\},
\end{equation}
noting that $S_0 \cup S_\infty$ is a discrete set that does not contain $0$.
Let $N_0=(S_0 \cup \Z) \cap (-\delta, \infty)$.
Similarly, let $N_\infty = (S_\infty \cup \Z) \cap (\delta, \infty)$.

The bound at $0$ implies that $\psi$ admits an asymptotic development around $|z|=0$ like 
\begin{equation}
 \psi \sim \sum_{\nu \in N_0 +\mathbb{N}}  r^\nu \zeta_\nu(\theta).
\end{equation}
As shown in Figure \ref{fig:indicialroot0}, because of our choice of $\delta$, $N_0 + \mathbb{N} \subset [0, \infty)$.
  \begin{figure}[ht]
  \begin{center}
   \includegraphics[width=2.5in]{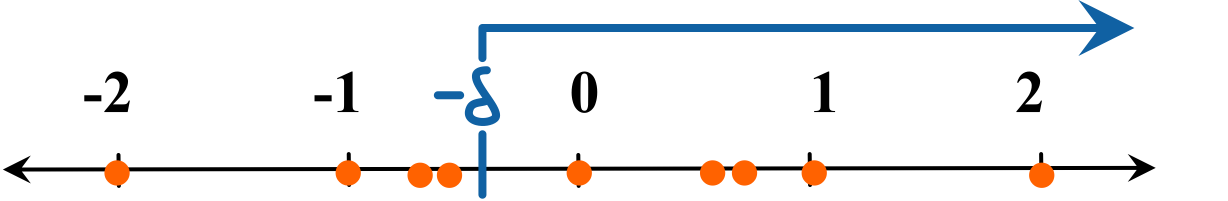} 
   \caption{\label{fig:indicialroot0}The indicial roots at $0$ are $S_0 \sqcup \Z$. With the given choice of $\delta$, the set of all powers in the expansion are nonnegative.} 
  \end{center}
 \end{figure}
 
 Furthering studying the $\star \IP{ \de_{\widetilde{A}} \psi, \psi}$ term in \eqref{eq:IBP3},
let $\widetilde{A} = \Xi(r) 2i \de \theta$. Then,  that
\begin{eqnarray}\label{eq:IBP2}
 \star \IP{\de_{\widetilde{A}} \psi, \psi} &=&
 \star  \IP{\del_r \psi \de r + \del_{\theta} \psi \de \theta + 2\I[\Xi, \psi] \de \theta, \psi}\\ \nonumber
 &=&   \IP{\del_r \psi, \psi} r \de \theta  - \IP{\del_\theta \psi + 2\I[\Xi, \psi],\psi} r^{-1} \de r.
\end{eqnarray}
To see that $\lim \limits_{r \to 0} \int_{S^1_r} \star \IP{\de_{\widetilde{A}} \psi, \psi}=0$,
note that the $\IP{\del_r \psi, \psi} r \de \theta$ term tends to zero pointwise. If we integrate on $S^1_r$, then the $\de r$ component does not matter.
However, we want to show that it's not problematic to perturb the loop $S^1_r$---even though the indicial roots $\nu \leq 1$ look problematic because of the ``$r^{-1}$''.
Since $\psi \in \Gamma(\I \mathfrak{su}(n))$, 
it is a straightforward computation to check that
\begin{equation} \label{eq:samereason}
  \IP{\del_\theta \psi + 2\I[\Xi, \psi],\psi} r^{-1} \de r = \IP{\del_\theta \psi, \psi} r^{-1} \de r.
\end{equation}
From the indicial root computation in Lemma \ref{lem:indicialroots}, we see that if $\nu<1$, then $\zeta'_\nu(\theta)=0$; thus the problematic-looking terms in \eqref{eq:IBP2} vanish.
It follows that $\lim \limits_{r \to 0} \int_{S^1_r} \star \IP{\de_{\widetilde{A}} \psi, \psi}=0$.

Similarly, in the coordinate $v=z^{-1} = s \e^{\I \vartheta}$, the bound on $\psi$ is given by $\psi(v) \leq |v|^\delta$.
As shown in Figure \ref{fig:indicialrootinf}, because of our choice of $\delta$, $N_\infty + \mathbb{N} \subset (0, \infty)$.
  \begin{figure}[ht]
  \begin{center}
   \includegraphics[width=2.5in]{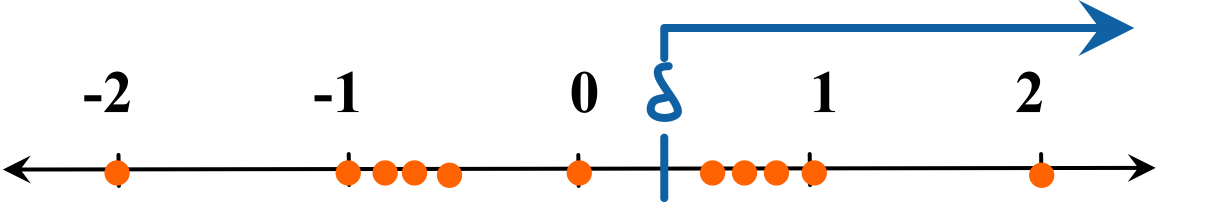} 
   \caption{\label{fig:indicialrootinf}The indicial roots at $\infty$ are $S_\infty \sqcup \Z$. With the given choice of $\delta$, the set $N_\infty$ of all permissible indicial roots are positive.} 
  \end{center}
 \end{figure}
Using a similar computation, we can see that $\lim \limits_{r \to \infty} \int_{S^1_r} \star \IP{\de_{\widetilde{A}} \psi, \psi}=0$.

Consequently, by \eqref{eq:IBP3}, we can equivalently show that there is no non-zero solution $\psi \in \Gamma(i \mathfrak{su}(n))$ solving
\begin{equation} \label{eq:equivalentproblem}
 \de_{\widetilde{A}} \psi =0 \qquad |\psi | \leq (\eps +  |z|^2)^{-\delta/2}
\end{equation}
where $\eps =0,1$. We now analyze now analyze this problem.
The $(i,j)$ block of $\de_{\widetilde{A}}\gamma$ is 
\begin{equation}
0=(\de_{\widetilde{A}} \gamma)_{ij} = \del_r \gamma_{ij} \de r + \del_\theta \gamma_{ij} \de \theta  + (f_i -f_j) \gamma_{ij} 2 \I \de \theta. 
\end{equation}
Note the following consequence:
\begin{equation} \label{eq:consequence}
 0 = \del_\theta ( \del_r \gamma_{ij}) = \del_r (\del_\theta \gamma_{ij}) = \del_r (2\I(f_j - f_i) \gamma_{ij}).
\end{equation}
Consequently, there are three distinct cases.
\begin{enumerate}
 \item If $f_j-f_i=0$, then $\gamma_{ij}$ is constant. Imposing the asymptotic decay condition, we see that $\gamma_{ij}=0$.
 \item If $f_j-f_i$ is not constant, it follows  
  that $\gamma_{ij}=0$ (using \eqref{eq:consequence} and $\del_r \gamma_{ij}=0$).
 \item If $f_j-f_i$ is a non-zero constant, then we see that we'd like to take  $\gamma_{ij}=c\e^{2\I \theta (f_j-f_i)}$. However, this is a function on the punctured-plane if, and only if $f_j-f_i\in \Z$.  However, note that for each function $f_i(r)$, $f_i(0) \in (-\frac{1}{2}, \frac{1}{2})$. Consequently, this case is not possible. 
\end{enumerate}
Note that in the last two cases we did not use the decay condition.

Thus, we've proved that there is no non-zero solution of \eqref{eq:equivalentproblem} and thereby no non-zero solution of \eqref{eq:nosoln}.
\end{proof}

\begin{rem}
 The proof is easier in the case where there are no blocks $\mathrm{A}_{\mathrm{mod}}$ appearing in $\widetilde{A}$. In this case, $\Delta_{\widetilde{A}}$ is a dilation-covariant 
 conic operator.
Solutions
are superpositions of functions of the form $r^\nu \zeta_\nu(\theta)$
where $\nu$ is an indicial root of the operator.
The bounds at $0$ force $\nu\geq 0$, while the bounds at $\infty$ force $\nu<0$.
These are incompatible, hence the zero solution is the only solution.
\end{rem}

\section{Perturbation to a solution of Hitchin's equations}\label{sec:main}

In this section we prove that $h_t^\app$ is close to the harmonic metric $h_t$ in the space of hermitian metrics.
We work in unitary formulation of Hitchin's equations, discussed at the beginning of  \S\ref{sec:linearization}.
\begin{figure}[ht]
\includegraphics[height=1in]{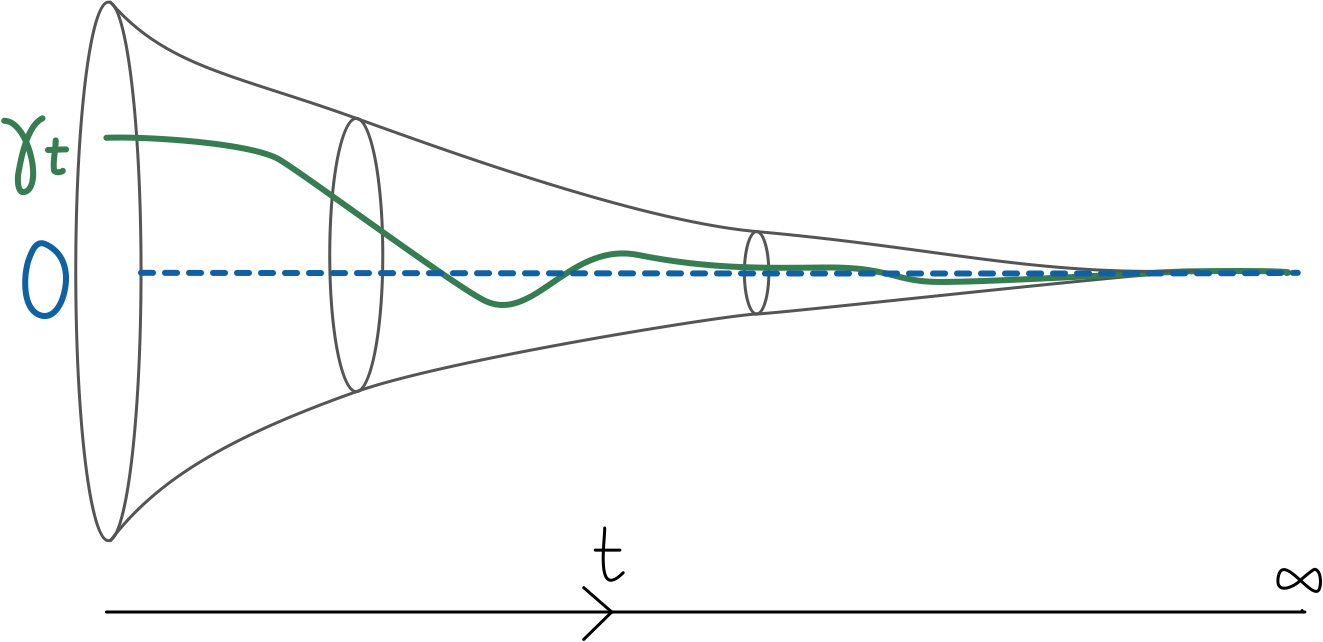}
\caption{\label{fig:mainthm}
Theorem \ref{thm:main} describes the size of $\gamma_t$ solving $\mathbf{F}_t^\app(\gamma_t)=0$, i.e. solving Hitchin's equations.
Alternatively, it describes the relation between $h_t$ and $h_t^\app$
in the space $\{ \mbox{hermitian metrics} \} \times \R_t^{+}$.  
In this language, the  dotted blue curve would be labeled $h_t^\app$ and
the solid green curve would be labeled $h_t$.
}
\end{figure}

\begin{mainthm}\label{thm:main}
Fix a Higgs bundle $(\delbar_E, \varphi) \in \cM'$ and let $\delta$ be the constant in Proposition \ref{defpropreg}.
Given any $\eps>0$ and choice of $t_0$, there exists a constant $C$ such that for $t>t_0$
there is a unique $h_0$-hermitian $\gamma_t$ satisfying
$\|\gamma_t\|_{H^2\left(i \mathfrak{su}(E)\right)} \leq C \e^{(-\delta + \eps) t}$,
such that  $\mathbf{F}^\app_t(\gamma_t)=0$, i.e. $(\de_{A_t^{\exp (-\gamma_t) }}, \Phi_t^{\exp(-\gamma_t)}))$ solves
Hitchin's equations.  (Equivalently, $h_t(v,w)=h_t^\app(\e^{-\gamma_t} v, \e^{-\gamma_t} w)$
is harmonic.)
Moreover, there exists a constant $C'$ such that $\gamma_t$ is unique in the the ball of radius $C' t^{-4-\eps}$.
\end{mainthm}

Theorem \ref{thm:main} will be  proved 
using a contraction mapping argument, as
in \cite{MSWW14}.
The map $\mathbf{F}_t^\app$ defined in \eqref{eq:Fapp} is naturally a map between the following Sobolev spaces
\begin{eqnarray}
 \mathbf{F}^\app_t: 
 H^2(i \mathfrak{su}(E)) &\rightarrow& L^2(i \mathfrak{su}(E)).
\end{eqnarray}
Observe that $\textbf{F}^\app_t(\gamma_t)=0$ if, and only if, 
$\gamma_t$ is a fixed point of the map
\begin{eqnarray}
\mathbf{T}_t: H^2(i \mathfrak{su}(E)) &\rightarrow& H^2(i \mathfrak{su}(E).\\ \nonumber
 \gamma &\mapsto& \gamma- L_t^{-1} 
\left(\textbf{F}^\app_t(\gamma)  \right).
\end{eqnarray}
Expanding $\mathbf{F}_t^\app$
into a constant, linear, and non-linear term
$ \textbf{F}^\app_t$
\begin{equation}\label{eq:sum}
 \mathbf{F}_t^\app(\gamma)=\mathbf{F}_t^\app(0)+L_t(\gamma) + Q_t(\gamma),
\end{equation}
the map $\mathbf{T}_t$ is
\begin{equation}\label{eq:T}
 \mathbf{T}_t(\gamma)=-(L_t)^{-1} (\mathbf{F}_t^\app(0) + Q_t(\gamma)).
\end{equation}
To show there is some ball $B_{\rho_t} \in H^2(i \mathfrak{su}(E))$ 
centered at the zero section (corresponding to $h_t^\app$) on which $\mathbf{T}_t$
is a contraction mapping of $B_{\rho_t}$,
we additionally need an estimate on the nonlinear terms in the expansion of the operator $\mathbf{T}_t$ in \eqref{eq:T}.

\subsection{Estimates for nonlinear terms}
We prove the analog of \cite[Lemma 6.8]{MSWW14} for the case of regular $SL(n,\C)$-Higgs bundles. 
\begin{lem}\label{lem:MSWW68}
The approximate solution satisfies
\begin{equation} \label{eq:appbound}
 \norm{A_t}_{C^1} \leq Ct
\end{equation}
on the disk $\bD$, so that for any $H^{k+1}$ section $\gamma$, $k=0,1$,
\begin{equation} \label{eq:bound12}
 \norm{\de_{A_t} \gamma }_{H^k} \leq Ct \norm{\gamma}_{H^{k+1}},
\end{equation}
and moreover,
\begin{equation} \label{eq:bound13}
\norm{L_t \gamma}_{L^2} \leq Ct^2 \norm{\gamma}_{H^2}.
\end{equation}
\end{lem}

\begin{proof}
We first prove the bound in \eqref{eq:appbound}. From Remark \ref{rem:regunitarygauge}, $A_t$ is diagonal with diagonal elements like
\begin{equation} \label{eq:diag1}
 -\frac{\alpha_{K,i}}{2} + \frac{|z|}{4} \frac{\de (u_{K,i,t} \chi) }{\de |z|}.
\end{equation}
This has the same asymptotics as 
\begin{equation}
 f_{i,t} :=  -\frac{\alpha_{K,i}}{2} + \frac{|z|}{4} \frac{\de u_{K,i,t}}{\de |z|},
\end{equation}
thus the diagonal elements of $A_t$ are uniformly bounded in $t$.
Since, we're computing the $C^1$ bound, we will
show that $|\del_{|z|} f_{i,t}| \leq Ct$.
Recall that $u_{K,i,t} = \zeta_t^*v_{K,i}$ where $v_{K,i}$
solve the affine Toda lattice in \eqref{eq:toda} and $\zeta_t(|z|)=\frac{2K}{K+1} |z|^{\frac{K+1}{K}}$.
Consequently,
\begin{equation}
 f_{i,t} = \zeta_t^* g_{K,i} \qquad g_{K,i}(\zeta) :=\left(-\frac{\alpha_{K,i}}{2} + \frac{(K+1)\zeta}{4K} \zeta \del_{\zeta} v_{K,i}(\zeta) \right).
\end{equation}
Note that $\lim_{\zeta \to 0} g'_{K,i}(\zeta) =0$ and 
$g'_{K,i}(\zeta)$ asymptotically decays like $\e^{-c_1 \zeta}$ from \eqref{eq:85b}.
Thus there is a constant $C_1>0$ such that for all $z$,
\begin{equation}
\left|f'_{i,t}(|z|)\right| = \left|g'_{K,i}(\zeta_t(|z|)) \cdot 2 t|z|^{\frac{1}{K}}\right| \leq C_1 t.
\end{equation}
Consequently the derivative of \eqref{eq:diag1} obeys a similar bound. The estimate in \eqref{eq:appbound} follows.

The bounds in \eqref{eq:bound12} and \eqref{eq:bound13} are immediate corollaries. For \eqref{eq:bound12}, observe that
\begin{equation}
 \norm{\de_{A_t} \gamma}_{H^k} \leq  \norm{\de \gamma}_{H^k}
 + \norm{[A_t, \gamma]}_{H^k}  \leq 
 \norm{\gamma}_{H^{k+1}}
 + \norm{A_t}_{C^k} \norm{\gamma}_{H^{k}} \leq C' t \norm{\gamma}_{H^{k+1}}.
\end{equation}
The bound in \eqref{eq:bound13} follows from \eqref{eq:bound12} and the bound on $M_{\Phi_t}$ in \eqref{eq:Mbound}.
\end{proof}

Using this lemma, we can derive the following estimate on the nonlinear terms in \eqref{eq:sum}.
\begin{lem}[Estimate on nonlinear terms] \cite[Lemma 6.9]{MSWW14} \label{lem:regnonlinear}
There exists a constant $\widehat{C}>0$ such that
\begin{equation}
 \norm{Q_t(\gamma_1) - Q_t(\gamma_2)}_{L^2} \leq \widehat{C} \xi t^2 \norm{\gamma_1 -\gamma_2}_{H^2}
\end{equation}
for all $\xi \in (0, 1]$ and $\gamma_1, \gamma_2$ satisfying $\norm{\gamma_i} \leq \xi$.
\end{lem}
\begin{proof}[Proof of Lemma \ref{lem:regnonlinear}]
The proof in \cite[Lemma 6.9]{MSWW14} carries over to the case of regular $SL(n,\C)$-Higgs bundles using Lemma \ref{lem:MSWW68} in place of \cite[Lemma 6.8]{MSWW14}.
\end{proof}

\subsection{Main Theorem}

\begin{proof}[Proof of Main Theorem \ref{thm:main}]
For all $\xi \in (0, 1]$ and $t>t_0$, hermitian sections $\gamma_1, \gamma_2$ satisfying $\norm{\gamma_i} \leq \xi$, 
\begin{eqnarray}
  \|\mathbf{T}_t(\gamma_1-\gamma_2)\|_{H^2} &=& \| -L_t^{-1} \left(Q_t(\gamma_1)- Q_t(\gamma_2) \right) \|_{H^2}\\ \nonumber
  &\leq& \|L_t^{-1} \|_{\mathcal{L}(L^2, H^2)} \| \|Q_t(\gamma_1)- Q_t(\gamma_2)\|_{L^2}\\ \nonumber
  & \leq & \tilde{C} t^2  \cdot \hat{C}\xi t^2 \|\gamma_1 - \gamma_2\|_{L^2}.
\end{eqnarray}
In the last line, we used the bounds in Lemma \ref{lem:regnonlinear} and Proposition \ref{prop:boundoninverse}.
Consequently setting $\xi$ equal to $r_t=\frac{1}{\tilde{C} \hat{C} t^4}$, $\mathbf{T}_t$ is a contraction
on the ball of radius $r_t$.

To see that there is a radius $\rho_t \leq r_t$ such that $\mathbf{T}_t(B_{\rho_t}) \subset
B_{\rho_t}$,
note that 
\begin{eqnarray}
  \|\mathbf{T}_t(\gamma)\|_{H^2} &\leq& \tilde{C} t^2  \cdot \hat{C}\xi t^2 \|\gamma\|_{L^2} + \|\mathbf{T}_t (0)\|_{H^2} \\ \nonumber
  &\leq& \tilde{C} t^2  \cdot \hat{C}\xi t^2 \|\gamma\|_{L^2} + \tilde{C}t^2 Ce^{-\delta t}.
\end{eqnarray}
Note that $\|\mathbf{T}_t(0)\|$ decays exponentially  in $t$ (Proposition \ref{defpropreg}) like $c\e^{-\delta t}$
for $t>t_0$.
Consequently, given any $\eps>0$, there exists a constant $C$ such that $\mathbf{T}_t(B_{\rho_t}) \subset
B_{\rho_t}$ for $\xi=\rho_t =C\e^{-(\delta+\eps)t}$.
Moreover, there exists a constant $C'$ such that 
$\mathbf{T}_t(B_{\rho'_t}) \subset
B_{\rho'_t}$ for $\xi=\rho'_t =C't^{-4-\eps}$. 

$\mathbf{T}_t$ is a contraction mapping on $B_{\rho_t}$ and $B_{\rho'_t}$.
Consequently, there
is exists a fixed point $\gamma_t$ of $\mathbf{T}_t$ in the ball $B_{C\e^{(-\delta+\eps)t}}$. Moreover, it is unique in the ball $B_{C't^{-4-\eps}}$.
\end{proof}

\begin{cor}\label{cor:main}
As $t \to \infty$, the harmonic metrics $h_t$ converge pointwise to $h_\infty$.
\end{cor}

\begin{proof}
 The approximate metrics $h_t^\app$ converge to $h_\infty$ pointwise by construction. The metric $h_t(v,w)=h_t^\app(\e^{-\gamma_t} v, \e^{-\gamma_t} w)$; consequently, because $\gamma_t \to 0$, $h_t$ also converge to $h_\infty$ pointwise.
\end{proof}

\begin{rem}
It's worth noting that this description of $h_t$ in terms of $h_t^\app$ does not work uniformly in $\cM'$.  As discussed in \S\ref{sec:stratification}, in our construction of $h_t^\app$, we have stratified $\cM'$.
Our construction works uniformly provided we stay in a single piece of the stratification.
However, the radius of the disks $\bD_p$ we use in the desingularization of $h_\ell$ goes to
zero when we pass between strata, as shown in Figure \ref{fig:stratification}.
  \begin{figure}[ht]
  \begin{center}
   \includegraphics[height=1in]{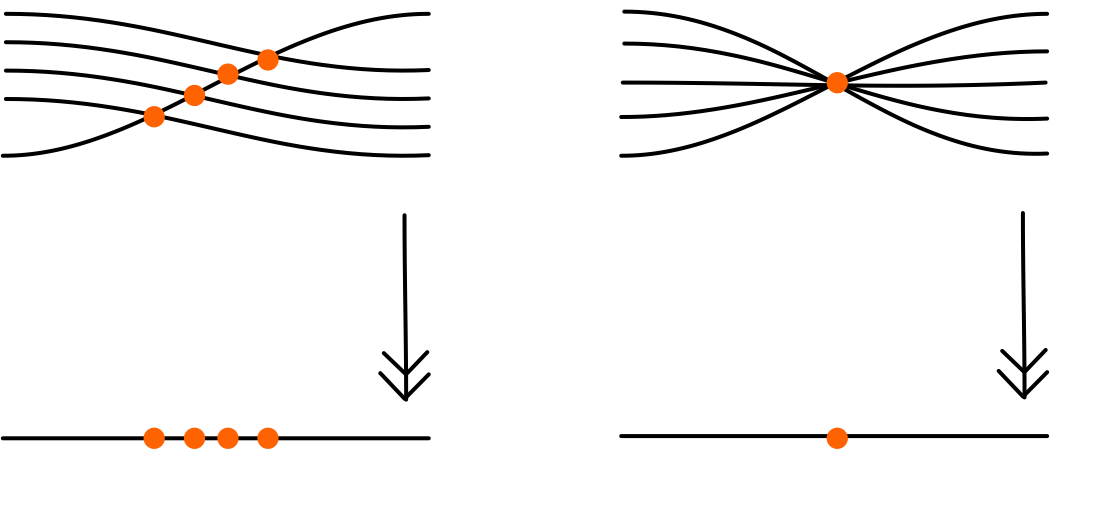} 
   \caption{\label{fig:stratification} Local model of spectral cover as one passed from (\textsc{Left}) $N_2=4$ to (\textsc{Right})$N_5=1$.} 
  \end{center}
 \end{figure}
Ideally, our construction would work uniformly on all of $\cM'$ since it should not be problematic
for ramification points to coalesce unless there is a collapsing cycle in the spectral cover, as shown in Figure \ref{fig:collapse}.
 \begin{figure}[ht]
  \begin{center}
      \includegraphics[height=1in]{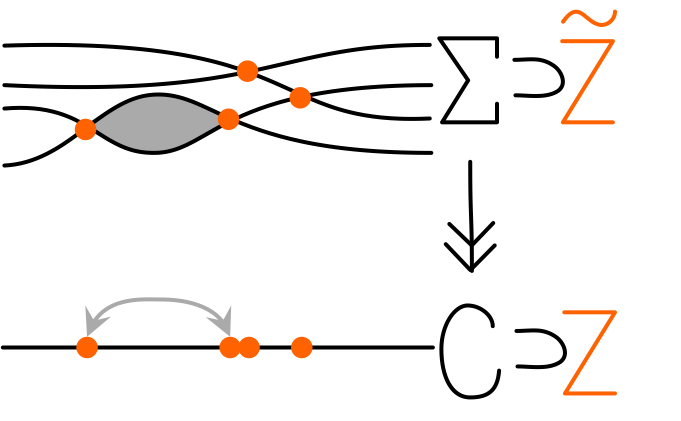}
   \caption{\label{fig:collapse}If the two indicated points of $Z$ coalesce, then a  cycle in $\Sigma$ collapses.} 
  \end{center}
 \end{figure}
\end{rem}
 
\bibliography{ends}{}
\bibliographystyle{fredrickson}

\end{document}